\documentclass[a4paper,11pt]{amsart}

\usepackage[french, english]{babel} 
\usepackage[T1]{fontenc}
\usepackage[ansinew]{inputenc}

\date{}
\title[Evolution equations associated with fractional Shubin operators]{Null-controllability of evolution equations associated with fractional Shubin operators through quantitative Agmon estimates}

\author{Paul Alphonse}
\address{Universit\'e de Lyon, ENSL, UMPA - UMR 5669, F-69364 Lyon}
\email{paul.alphonse@ens-lyon.fr}


\makeatletter
	\@namedef{subjclassname@2020}{\textup{2020} Mathematics Subject Classification}
\makeatother


\usepackage[top=3cm, bottom=2cm, left=3cm, right=3cm]{geometry}

\usepackage{enumitem}

\usepackage{array}										

\usepackage{amsfonts}
\usepackage{amsmath}
\usepackage{amsthm}
\usepackage{amssymb}

\usepackage{bbm}

\usepackage{stmaryrd}

\usepackage[scr]{rsfso}

\usepackage{comment}

\usepackage[dvipsnames]{xcolor}

\usepackage{hyperref}

\hypersetup{	
colorlinks=true,
breaklinks=true,
urlcolor= red,
linkcolor= red,
citecolor=ForestGreen
}

\numberwithin{equation}{section}
										
\newtheorem{thm}{Theorem}[section]
\newtheorem{prop}[thm]{Proposition}

\newtheorem{lem}[thm]{Lemma}
\newtheorem{cor}[thm]{Corollary}
\newtheorem{dfn}[thm]{Definition}
\theoremstyle{definition}

\DeclareMathOperator{\ad}{ad}
\DeclareMathOperator{\Leb}{Leb}
\DeclareMathOperator{\Supp}{Supp}

\DeclareMathOperator{\Vect}{Vect}

\begin{document}

\sloppy

\selectlanguage{english}

\begin{abstract} We consider the anisotropic Shubin operators $(-\Delta)^m + \vert x\vert^{2k}$ acting on the space $L^2(\mathbb R^n)$, with $k,m\geq1$ some positive integers. We provide sharp quantitative estimates in  Gelfand-Shilov spaces for the eigenfunctions of these selfadjoint differential operators with a strategy based on the classical approach to obtain Agmon estimates in spectral theory. By using a Weyl law for the eigenvalues of the anisotropic Shubin operators, we also describe the smoothing properties of the semigroups generated by the fractional powers of these operators, with precise estimates in short times. This description allows us to prove positive null-controllability results  for the associated evolution equations posed on the whole space $\mathbb R^n$, from control supports which are thick with respect to densities and in any positive time. We generalize in particular results known for the evolution equations associated with fractional harmonic oscillators.
\end{abstract}

\maketitle

\section{Introduction}
This paper is devoted in studying the smoothing properties and the null-controllability of the evolution equations associated with fractional anisotropic Shubin operators $H^s_{k,m}$ and posed on the whole space $\mathbb R^n$. These non-local operators $H^s_{k,m}$ are defined through the functional calculus as the fractional powers of the following anisotropic selfadjoint elliptic operators
\begin{equation}\label{19022020E1}
	H_{k,m} = (-\Delta)^m + \vert x\vert^{2k},\quad x\in\mathbb R^n,
\end{equation}
which we consider equipped with the domains
\begin{equation}\label{25092020E1}
	D(H_{k,m}) = \big\{g\in L^2(\mathbb R^n) : H_{k,m}g\in L^2(\mathbb R^n)\big\},
\end{equation}
where $k,m\geq1$ are two positive integers and $s>0$ is a positive real number. These operators naturally arise in physical models. For example, the fractional harmonic oscillator $H^s_{1,1}$ appears in the kinetic theory of gases \cite{LMPSX1, LMPSX2, LMPSX3}. Another example is given by the quantum anharmonic oscillators $H_{1,k}$ involved in quantum mechanics.

The study of the \textit{null-controllability} of evolution equations posed on the whole space $\mathbb R^n$, of elliptic type or degenerate of hypoelliptic type, and also the Schr\"odinger counterparts of such equations, has been much addressed in the last years \cite{A, AB, BEPS, BJKPS, MR3732691, K, Ko, MPS, MPS2}. Although considerable progress have already been made, the understanding of these equations is still at an early stage, in opposite to the same models posed on bounded domains of $\mathbb R^n$, for which many behaviors have been highlighted, see e.g. the introduction of \cite{MR3732691}. In this work, we tackle null-controllability issues for the following evolution equations associated with fractional anisotropic Shubin operators
\begin{equation}\label{27112020E1}\tag{$E_{s,k,m}$}
\left\{\begin{aligned}
	& \partial_tf(t,x) + H^s_{k,m}f(t,x) = h(t,x)\mathbbm 1_{\omega}(x),\quad t>0,\ x\in\mathbb R^n, \\
	& f(0,\cdot) = f_0\in L^2(\mathbb R^n).
\end{aligned}\right.
\end{equation}
On the one hand, we prove that the equation \eqref{27112020E1} is null-controllable from \textit{thick} control supports $\omega\subset\mathbb R^n$ in any positive time $T>0$, under the large diffusion assumption $2sm>1$. The notion of thickness has appeared to be central in the null-controllability theory since the works \cite{MR3816981, WZ}, where the authors established that this is a necessary and sufficient geometric condition that ensures the null-controllability of the heat equation posed on $\mathbb R^n$. The same phenomena holds true more generally for the evolution equations associated with fractional Laplacians $(-\Delta)^s$ under the same setting and when $s>1/2$, as proven in \cite{AB}, and also quite surprisingly for the Schr\"odinger counterpart of this equation in the one dimensional setting and when $s\geq1/2$, see \cite{MPS2}. It is also known from \cite{K, Ko} that in the cases $0<s\le1/2$, the fractional heat equations are not null-controllable from thick control supports anymore. In the recent work \cite{AM}, the notion of thickness has appeared to be a necessary and sufficient condition to ensure the stabilization or the approximate null-controllability with uniform cost (which are notions weaker than the null-controllability) of a very large class of diffusive equations posed on $\mathbb R^n$, including in particular the half heat equation associated with the operator $(-\Delta)^{1/2}$. Finally, let us mention that other classes of degenerate parabolic equations of hypoelliptic type, as evolution equations associated with accretive quadratic operators or (non-autonomous) Ornstein-Uhlenbeck operators, were proven to be null-controllable from thick control supports, see \cite{A, BEPS, BJKPS}. On the other hand, we establish that in the isotropic case where $k=m=l$, the equation (\hyperref[01202020E1]{$E_{s,l,l}$}) is null-controllable in any positive time $T>0$ from control supports which are \textit{thick with respect to densities}. This notion, which is an extension of the thickness property, was introduced in the work \cite{MPS} in order to tackle null-controllability issues for evolution equations enjoying strong smoothing properties in symmetric Gelfand-Shilov spaces. In particular, we generalize a result from \cite{MPS} concerning the null-controllability of fractional heat harmonic equations. Finally, we prove that in the more specific case $k=m=1$ and $s>1$, the equation (\hyperref[01202020E1]{$E_{s,1,1}$}) is always null-controllable in any positive time $T>0$ from any support control $\omega\subset\mathbb R^n$ which is measurable with positive Lebesgue measure.

These null-controllability issues motivate the study of the \textit{smoothing properties} of semigroups generated by selfadjoint or non-seladjoint accretive operators, which is also natural and interesting in itself \cite{A, AB,AB2, BEPS}. The major part of the present work consists in fact in describing the regularizing effects of the semigroups generated by fractional anisotropic Shubin operators $H^s_{k,m}$ on $L^2(\mathbb R^n)$. Precisely, we prove that the evolution operators generated by these operators enjoy smoothing properties in \textit{Gelfand-Shilov spaces} 
in any positive time $t>0$, 
$$\forall t>0, \forall g\in L^2(\mathbb R^n),\quad e^{-tH^s_{k,m}}g\in S^{\mu_{s,k,m}}_{\nu_{s,k,m}}(\mathbb R^n),$$
with the regularity exponents $\nu_{s,k,m}>0$ and $\mu_{s,k,m}>0$ given by
$$\nu_{s,k,m} = \max\bigg(\frac1{2sk},\frac m{k+m}\bigg)\quad\text{and}\quad\mu_{s,k,m} = \max\bigg(\frac1{2sm},\frac k{k+m}\bigg),$$
by providing the following quantitative estimates for the associated seminorms in short times $0<t\ll1$, 
$$\big\Vert x^{\alpha}\partial^{\beta}_x(e^{-tH^s_{k,m}}g)\big\Vert_{L^2(\mathbb R^n)}\le\frac{C^{\vert\alpha\vert+\vert\beta\vert}}{t^{\nu_{s,k,m}\vert\alpha\vert+\mu_{s,k,m}\vert\beta\vert+\frac{n(k+m)}{2skm}}}\, (\alpha!)^{\nu_{s,k,m}}\, (\beta!)^{\mu_{s,k,m}}\, \Vert g\Vert_{L^2(\mathbb R^n)}.$$
The strategy consists in first obtaining the following sharp quantitative \textit{Agmon estimates} for the eigenfunctions associated with the anisotropic Shubin operators as follows
$$\big\Vert e^{c_1t\langle x\rangle^{\sigma(1+\frac km)}}\psi\big\Vert_{L^2(\mathbb R^n)} + \big\Vert e^{c_1t\langle D_x\rangle^{\sigma(1+\frac mk)}}\psi\big\Vert_{L^2(\mathbb R^n)}\le c_2e^{c_2t\lambda^{\sigma(\frac1{2k}+\frac1{2m})}}\Vert\psi\Vert_{L^2(\mathbb R^n)},$$
where $\lambda>0$ is eigenvalue associated with the eigenfunction $\psi$ of the operators $H_{k,m}$ and $0\le t\le T$, $0\le\sigma\le1$ are some parameters. These Agmon estimates combined with a Weyl law from \cite{BBR} then allows to obtain the above smoothing properties of the semigroups generated by the fractional anisotropic Shubin operators. From a spectral point of view, the anisotropic Shubin operators have been widely studied in the last decades, from the work \cite{V} on the quartic oscillator $H_{1,4}$ and the works \cite{HR1, HR2} on a class of anharmonic oscillators containing the quantum harmonic oscillators $H_{1,k}$, or the papers \cite{H,R} considering the symmetric case $k=l$. The regularity of the eigenfunctions of general anisotropic Shubin operators has already been studied qualitatively in the work \cite{MR2747070}. Let us also mention the paper \cite{CDR} in which a general class of anisotropic Shubin operators is studied within the framework of the Weyl-H\"ormander calculus and where spectral properties in terms of Schatten-von Neumann classes for the negative powers of these operators are obtained.

\subsubsection*{Outline of the work} In Section \ref{results}, we present in details the main results contained in this work. Section \ref{nullcont} is devoted to the proofs of the positive null-controllability results for the evolution equations associated with fractional anisotropic Shubin operators. Quantitative Agmon estimates for the eigenfunctions of these operators are obtained in Section \ref{agmon}, which allow to describe the smoothing properties of the semigroups generated by their fractional powers in Section \ref{smoothing}. The proof of these Agmon estimates require a technical G{\aa}rding type inequality obtained in Section \ref{secgard}. Finally, basics of Gelfand-Shilov spaces are presented in Section \ref{appendix}, which is an Appendix also containing a microlocal result dealing with the density of the Schwartz space in the graph of the differential operators.

\subsubsection*{Notations} The following notations and conventions will be used all over the work:
\begin{enumerate}[label=\textbf{\arabic*.},leftmargin=* ,parsep=2pt,itemsep=0pt,topsep=2pt]
\item The canonical Euclidean scalar product of $\mathbb R^n$ is denoted by $\cdot$ and $\vert\cdot\vert$ stands for the associated canonical Euclidean norm. The Japanese bracket $\langle\cdot\rangle$ is defined for all $x\in\mathbb R^n$ by $\langle x\rangle = \sqrt{1+\vert x\vert^2}$.
\item For all measurable subset $\omega\subset\mathbb R^n$, the inner product of $L^2(\omega)$ is defined by
$$\langle u,v\rangle_{L^2(\omega)} = \int_{\omega}u(x)\overline{v(x)}\ \mathrm dx,\quad u,v\in L^2(\omega),$$
while $\Vert\cdot\Vert_{L^2(\omega)}$ stands for the associated norm.
\item For all function $u\in\mathscr S(\mathbb R^n)$, the Fourier transform of $u$ is defined by
$$\widehat u(\xi) = \int_{\mathbb R^n}e^{-ix\cdot\xi}u(x)\ \mathrm dx.$$
With this convention, Plancherel's theorem states that 
$$\forall u\in L^2(\mathbb R^n),\quad \Vert\widehat u\Vert_{L^2(\mathbb R^n)} = (2\pi)^{n/2}\Vert u\Vert_{L^2(\mathbb R^n)}.$$
\item We denote the gradient by $\nabla_x$ and the Laplacian operator by $\Delta$. Moreover, we set $D_x = -i\nabla_x$ and for all $q>0$, we define by $\langle D_x\rangle^q$ the Fourier multiplier associated with the symbol $\langle\xi\rangle^q$.
\item We use the notation $H^q(\mathbb R^n)$ for the Sobolev spaces, with $q\geq0$ non-negative real numbers, and we denote by $\dot H^q(\mathbb R^n)$ their homogeneous counterparts.
\item The space $C^{\infty}_b(\mathbb R^n)$ stands for the set of smooth functions $g\in C^{\infty}(\mathbb R^n)$ with bounded derivatives.
\item For all measurable subset $\omega\subset\mathbb R^n$, $\mathbbm{1}_{\omega}$ stands for the characteristic function of $\omega$.
\end{enumerate}

\section{Statement of the main results}\label{results}

This section is devoted in presenting in details the main results contained in this work. Let us begin by quickly recalling the definition of the fractional powers of the operator $H_{k,m}$, defined in \eqref{19022020E1} and equipped with the domain \eqref{25092020E1}, that we will consider in the following. Since the operator $H_{k,m}$ is a positive anisotropic elliptic operator, there exists a Hilbert basis $(\psi_j)_j$ of $L^2(\mathbb R^n)$ composed of eigenfunctions of the operator $H_{k,m}$, see e.g. \cite{S}. Moreover, denoting $\lambda_j>0$ the eigenvalue associated with the eigenfunction $\psi_j\in L^2(\mathbb R^n)$, the family $(\lambda_j)_j$ satisfies $\lim_j\lambda_j = +\infty$. Given $s>0$ a positive real number, one can define the operator $H^s_{k,m}$ in the following way
\begin{equation}\label{20112020E3}
	\forall g\in D(H^s_{k,m}),\quad H^s_{k,m}g = \sum_{j=0}^{+\infty}\lambda_j^s\langle g,\psi_j\rangle_{L^2(\mathbb R^n)}\psi_j,
\end{equation}
equipped with  the domain
\begin{equation}\label{20112020E4}
	D(H^s_{k,m}) = \bigg\{g\in L^2(\mathbb R^n) : \sum_{j=0}^{+\infty}\lambda_j^{2s}\big\vert\langle g,\psi_j\rangle_{L^2(\mathbb R^n)}\big\vert^2<+\infty\bigg\}.
\end{equation}
Notice that the above domain coincides with \eqref{25092020E1} in the case where $s=1$, from Parseval's formula. The fractional anisotropic Shubin operator $H^s_{k,m}$ is then a selfadjoint operator that generates a strongly continuous semigroup on $L^2(\mathbb R^n)$ explicitly given by
$$\forall t\geq0, \forall g\in L^2(\mathbb R^n),\quad e^{-tH^s_{k,m}}g = \sum_{j=0}^{+\infty}e^{-t\lambda_j^s}\langle g,\psi_j\rangle_{L^2(\mathbb R^n)}\psi_j,$$
see e.g. \cite{TW} (Propositions 2.6.2 and 2.6.5).

\subsection{Quantitative Agmon estimates and smoothing properties} First of all, we aim at understanding the smoothing properties enjoyed by the semigroup generated by the fractional anisotropic Shubin operator $H^s_{k,m}$. In order to carry out this study, we begin by establishing Agmon estimates for the eigenfunctions associated with the operator $H_{k,m}$. Generally, Agmon estimates, which originate in the pioneer works \cite{Ag1, Ag2}, aim at quantifying the exponential decaying properties of eigenfunctions associated with some large classes of selfadjoint operators.

\begin{thm}\label{19022020T1} Let $k,m\geq1$ be positive integers and $H_{k,m}$ be the associated anisotropic Shubin operator defined in \eqref{19022020E1} and equipped with the domain \eqref{25092020E1}. We also consider $0\le\sigma\le1$ a non-negative real number. There exist some positive constants $c_1,c_2>0$ and $T>0$ such that for all eigenfunction $\psi\in L^2(\mathbb R^n)$ of the operator $H_{k,m}$ and $0\le t\le T$,
$$\big\Vert e^{c_1t\langle x\rangle^{\sigma(1+\frac km)}}\psi\big\Vert_{L^2(\mathbb R^n)} + \big\Vert e^{c_1t\langle D_x\rangle^{\sigma(1+\frac mk)}}\psi\big\Vert_{L^2(\mathbb R^n)}\le c_2e^{c_2t\lambda^{\sigma(\frac1{2k}+\frac1{2m})}}\Vert\psi\Vert_{L^2(\mathbb R^n)},$$
with $\lambda>0$ the eigenvalue associated with the eigenfunction $\psi$.
\end{thm}

This result allows to recover the Gelfand-Shilov regularity of the eigenfunctions of the operator $H_{k,m}$, which is a consequence of Theorem 1.1 in \cite{MR2747070}. We refer to Subsection \ref{GS} in Appendix for a definition of the Gelfand-Shilov spaces $S^{\mu}_{\nu}(\mathbb R^n)$ and their basic properties, with $\mu,\nu>0$ some positive real numbers satisfying $\mu+\nu\geq1$. The merit of Theorem \ref{19022020T1} is to provide a quantitative description of the regularity of those eigenfunctions. Its proof is based on the classical strategy employed to obtain Agmon estimates in spectral theory, which requires in our context to obtain a quite technical G{\aa}rding type inequality. Combined to a Weyl law for the eigenvalues of the anisotropic Shubin operators, Theorem \ref{19022020T1} allows to describe the smoothing properties of the evolution operators generated by  the fractional operator $H^s_{k,m}$.

\begin{cor}\label{02102020C1} Let $k,m\geq1$ be positive integers, $s>0$ be a positive real number, and $H^s_{k,m}$ be the associated fractional anisotropic Shubin operator defined in \eqref{20112020E3} and equipped with the domain \eqref{20112020E4}. There exist some positive constants $c_1,c_2>0$ and $0<T<1$ such that for all $0<t<T$ and $g\in L^2(\mathbb R^n)$,
$$\big\Vert e^{c_1t\langle x\rangle^{\frac1{\nu_{s,k,m}}}}(e^{-tH^s_{k,m}}g)\big\Vert_{L^2(\mathbb R^n)} + \big\Vert e^{c_1t\langle D_x\rangle^{\frac1{\mu_{s,k,m}}}}(e^{-tH^s_{k,m}}g)\big\Vert_{L^2(\mathbb R^n)}\le\frac{c_2}{t^{\frac{n(k+m)}{2skm}}}\Vert g\Vert_{L^2(\mathbb R^n)},$$
the exponents $\nu_{s,k,m}>0$ and $\mu_{s,k,m}>0$ being given by
$$\nu_{s,k,m} = \max\bigg(\frac1{2sk},\frac m{k+m}\bigg)\quad\text{and}\quad\mu_{s,k,m} = \max\bigg(\frac1{2sm},\frac k{k+m}\bigg).$$
\end{cor}

Notice that Corollary \ref{02102020C1} and Lemma \ref{02012020L1} also imply that there exists a positive constant $C>0$ such that for all $0<t<T$, $(\alpha,\beta)\in\mathbb N^{2n}$ and $g\in L^2(\mathbb R^n)$,
\begin{equation}\label{20112020E5}
	\big\Vert x^{\alpha}\partial^{\beta}_x(e^{-tH^s_{k,m}}g)\big\Vert_{L^2(\mathbb R^n)}\le\frac{C^{\vert\alpha\vert+\vert\beta\vert}}{t^{\nu_{s,k,m}\vert\alpha\vert+\mu_{s,k,m}\vert\beta\vert+\frac{n(k+m)}{2skm}}}\ (\alpha!)^{\nu_{s,k,m}}\ (\beta!)^{\mu_{s,k,m}}\ \Vert g\Vert_{L^2(\mathbb R^n)}.
\end{equation}
This result highlights the existence of a critical diffusion index $0<s_{cr}\le 1$, given by
\begin{equation}\label{19112020E1}
	s_{cr} = \frac1{2k}+\frac1{2m},
\end{equation}
for which the semigroup generated by the operator $H^s_{k,m}$ enjoys different smoothing properties, depending on whether $s\le s_{cr}$ or $s>s_{cr}$. The existence of this critical index is in fact a consequence of an uncertainty principle. Indeed, when $s\le s_{cr}$, according to \eqref{20112020E5}, the evolution operators generated by the operator $H^s_{k,m}$ enjoy the following smoothing properties
\begin{equation}\label{20112020E6}
	\forall t>0, \forall g\in L^2(\mathbb R^n),\quad e^{-tH^s_{k,m}}g\in S^{1/2sm}_{1/2sk}(\mathbb R^n).
\end{equation}
Moreover, Theorem \ref{20112020T1} stated in Appendix, which can be read as a version of the Heisenberg's uncertainty principle, shows that the Gelfand-Shilov space involved in \eqref{20112020E6} is not reduced to zero provided $1/2sk + 1/2sm\geq1$, that is, $s\le s_{cr}$. As a consequence, when $s>s_{cr}$, the property \eqref{20112020E6} cannot hold anymore and the estimates \eqref{20112020E5} imply that
\begin{equation}\label{20112020E7}
	\forall t>0, \forall g\in L^2(\mathbb R^n),\quad e^{-tH^s_{k,m}}g\in S^{k/(k+m)}_{m/(k+m)}(\mathbb R^n).
\end{equation}
Roughly speaking, the Gelfand-Shilov smoothing properties of the evolution operators generated by the operator $H^s_{k,m}$ ``are stationary from $s = s_{cr}$''.

As established in \cite{MR3996060} (Theorem 1.4) and presented in Subsection \ref{GS} in Appendix (more precisely in \eqref{15102020E1}), the Gelfand-Shilov spaces $S^{\mu}_{\nu}(\mathbb R^n)$, with $\mu/\nu\in\mathbb Q$, can be characterized through the decomposition into the basis of eigenfunctions of anisotropic Shubin operators. By using this property, one could check that the qualitative properties \eqref{20112020E6} and \eqref{20112020E7} hold. However, we absolutely need to use the Agmon quantitative estimates provided by Theorem \ref{19022020T1} to obtain the quantitative Gelfand-Shilov smoothing properties stated in Corollary \ref{02102020C1}, which are requested to prove the null-controllability results stated in the Subsection \ref{Null-Cont}. 

The presence of the term $t^{-n(k+m)/2skm}$ in the right-hand side of the inequalities presented in Corollary \ref{02102020C1} was not expected and we conjecture that estimates of the following form hold
\begin{equation}\label{20112020E8}
	\big\Vert e^{c_1t\langle x\rangle^{\frac1{\nu_{s,k,m}}}}(e^{-tH^s_{k,m}}g)\big\Vert_{L^2(\mathbb R^n)} + \big\Vert e^{c_1t\langle D_x\rangle^{\frac1{\mu_{s,k,m}}}}(e^{-tH^s_{k,m}}g)\big\Vert_{L^2(\mathbb R^n)}\le c_2\Vert g\Vert_{L^2(\mathbb R^n)}.
\end{equation}
In fact, by adapting arguments used in the proof of Theorem \ref{19022020T1}, we can prove this conjecture in the special case $s=1$.

\begin{thm}\label{14102020E1} Let $m,k\geq1$ be positive integers and $H_{k,m}$ be the associated anisotropic Shubin operator defined in \eqref{19022020E1} and equipped with the domain \eqref{25092020E1}.
There exist some positive constants $c_1,c_2>0$ and $0<T<1$ such that for all $0\le t\le T$ and $g\in L^2(\mathbb R^n)$,
$$\big\Vert e^{c_1t\langle x\rangle^{1+\frac km}}(e^{-tH_{k,m}}g)\big\Vert_{L^2(\mathbb R^n)} + \big\Vert e^{c_1t\langle D_x\rangle^{1+\frac mk}}(e^{-tH_{k,m}}g)\big\Vert_{L^2(\mathbb R^n)}\le c_2\Vert g\Vert_{L^2(\mathbb R^n)}.$$
\end{thm}

We also expect that the strategy employed to prove Theorem \ref{14102020E1} can be adapted to obtain the estimates \eqref{20112020E8}. However, this would \textit{a priori} require to obtain a G{\aa}rding type inequality far more technical than one obtained while proving Theorem \ref{19022020T1}. Since the result given by Corollary \ref{02102020C1} will be sufficient to obtain null-controllability results, we will not tackle such a generalization in this work.

\subsection{Null-controllability}\label{Null-Cont} As an application of Corollary \ref{02102020C1}, we therefore study the null-controllability of the evolution equations associated with fractional anisotropic Shubin operators and posed on the whole space $\mathbb R^n$. More precisely, for all positive integers $k,m\geq1$ and all positive real number $s>0$, we consider the equation
\begin{equation}\label{01202020E1}\tag{$E_{s,k,m}$}
\left\{\begin{aligned}
	& \partial_tf(t,x) + H^s_{k,m}f(t,x) = h(t,x)\mathbbm 1_{\omega}(x),\quad t>0,\ x\in\mathbb R^n, \\
	& f(0,\cdot) = f_0\in L^2(\mathbb R^n),
\end{aligned}\right.
\end{equation}
where $H^s_{k,m}$ is the fractional anisotropic Shubin operator defined in \eqref{20112020E3} and equipped with the domain \eqref{20112020E4}, and $\omega\subset\mathbb R^n$ is a Borel set with a positive Lebesgue measure. 

\begin{dfn}\label{24112020D1} The equation \eqref{01202020E1} is said to be null-controllable from the support control $\omega$ in time $T>0$ if, for all initial datum $f_0\in L^2(\mathbb R^n)$, there exists a control $h\in L^2((0,T)\times\omega)$ such that the mild solution of \eqref{01202020E1} satisfies $f(T,\cdot) = 0$. 
\end{dfn}

The precise knowledge of the smoothing properties of the semigroup generated by the fractional anisotropic operator $H^s_{k,m}$ plays a key in the study of the null-controllability of the evolution equation \eqref{01202020E1}. This link, made by the Hilbert Uniqueness Method and the Lebeau-Robbiano strategy, will be presented in details in Section \ref{nullcont}. For now, we only present the null-controllability results contained in this work.

We begin by studying the null-controllability of the evolution equation \eqref{01202020E1} from thick control supports $\omega\subset\mathbb R^n$.

\begin{dfn} A set $\omega\subset\mathbb R^n$ is called $\gamma$-thick at scale $L>0$, with $\gamma\in(0,1]$, if it is measurable and satisfies
$$\forall x\in\mathbb R^n,\quad\Leb(\omega\cap(x + [0,L]^n))\geq\gamma L^n,$$
where $\Leb$ stands for the Lebesgue measure on $\mathbb R^n$. A set $\omega\subset\mathbb R^n$ is then called thick when there exist $\gamma\in(0,1]$ and $L>0$ such that $\omega$ is $\gamma$-thick at scale $L$.
\end{dfn}

First of all, we prove that the evolution equations \eqref{01202020E1} are null-controllable from thick control supports $\omega\subset\mathbb R^n$ in any positive time $T>0$, under an assumption of large diffusion.

\begin{thm}\label{09112020T1} Let $k,m\geq1$ be positive integers and $s>0$  be a positive real number satisfying $2sm>1$. When the control support $\omega\subset\mathbb R^n$ is a thick set, the evolution equation \eqref{01202020E1} is null-controllable from $\omega$ in any positive time $T>0$.
\end{thm}

The assumption $2ms>1$ in Theorem \ref{09112020T1} has to be related to the assumption $s>1/2$ mentioned above that ensures the null-controllability of fractional heat equations from thick control supports in any positive time, which formally correspond to the equations (\hyperref[01202020E1]{$E_{s,0,1}$}). However, in opposite to the fractional heat equations, it is still an open and interesting equation to investigate if the equations \eqref{01202020E1} are null-controllable from thick sets in the low diffusion regime $0<2sm\le1$.

Let us mention that we will obtain a more quantitative result than the one stated in Theorem \ref{09112020T1}. Precisely, considering a thick set $\omega\subset\mathbb R^n$, we will prove the following observability estimate: there exists a positive constant $C>1$ such that for all $T>0$ and $g\in L^2(\mathbb R^n)$,
\begin{equation}\label{20112020E1}
	\big\Vert e^{-TH^s_{k,m}}g\big\Vert^2_{L^2(\mathbb R^n)}\le C\exp\bigg(\frac C{T^{\beta_{s,k,m}}}\bigg)\int_0^T\big\Vert e^{-tH^s_{k,m}}g\big\Vert^2_{L^2(\omega)}\,\mathrm dt,
\end{equation}
with
$$\beta_{s,k,m} = \max\bigg(\frac1{2sm-1},\frac km\bigg).$$
The notions of null-controllability and observability are equivalent in our context by the Hilbert Uniqueness Method, as explained in the beginning of Section \ref{nullcont}. Such observability estimates have already been obtained in the works \cite{MR2984079, M, MR2679651} in the particular case where $m = 1$, $k\geq2$ and $s=1$. Precisely, \cite{MR2984079} (Theorem 1.10) states that when $k\geq2$ and $\Gamma\subset\mathbb R^n$ is a cone of the form 
\begin{equation}\label{20112020E2}
	\Gamma = \big\{x\in\mathbb R^n : \vert x\vert>r_0,\ x/\vert x\vert\in\Omega_0\big\},
\end{equation}
where $r_0>0$ and $\Omega_0$ is an open subset of the unit sphere of $\mathbb R^n$, there exists a positive constant $C_0>0$ such that for all $T>0$ and $g\in L^2(\mathbb R^n)$,
$$\big\Vert e^{-TH_{k,1}}g\big\Vert^2_{L^2(\mathbb R^n)}\le\exp\bigg(\frac{C_0}{T^{1+\frac2{k-1}}}\bigg)\int_0^T\big\Vert e^{-tH_{k,1}}g\big\Vert^2_{L^2(\Gamma)}\,\mathrm dt.$$
Since the notions of null-controllability and observability are equivalent for the equations we are dealing with, as already mentioned, the above estimate implies that when $k\geq2$, the equation (\hyperref[01202020E1]{$E_{1,k,1}$}) is null-controllable in any positive times $T>0$ from any cone of the form \eqref{20112020E2}. Notice that under the same setting, for comparison, the observability estimate \eqref{20112020E1} writes
$$\big\Vert e^{-TH_{k,1}}g\big\Vert^2_{L^2(\mathbb R^n)}\le C\exp\bigg(\frac C{T^k}\bigg)\int_0^T\big\Vert e^{-tH_{k,1}}g\big\Vert^2_{L^2(\omega)}\,\mathrm dt,$$
and also that $1+2/(k-1)\le k$ provided $k\geq3$. We will come back to the result \cite{MR2984079} (Theorem 1.10) a little further.

As it will appear in Subsection \ref{ani}, the proof of Theorem \ref{09112020T1} only uses the Gevrey smoothing properties given by Corollary \ref{02102020C1}, that is, the estimate 
$$\big\Vert e^{c_1t\vert D_x\vert^{\frac1{\mu_{s,k,m}}}}(e^{-tH^s_{k,m}}g)\big\Vert_{L^2(\mathbb R^n)}\le\frac{c_2}{t^{\frac{n(k+m)}{2skm}}}\Vert g\Vert_{L^2(\mathbb R^n)},$$
and not the full Gelfand-Shilov smoothing properties provided by the same result. In the general case where $k,m\geq1$ are arbitrary, it is an open and very interesting problem to know how obtaining positive null-controllability results for the equation \eqref{01202020E1} by also using the exponentiel decay properties of the semigroup generated by the operator $H^s_{k,m}$. However, the result given by Theorem \ref{09112020T1} can be extended in the isotropic case where $k=m$, by considering a more general class of control supports, those which are thick with respect to densities, introduced in the work \cite{MPS}.

\begin{dfn} Let $\omega\subset\mathbb R^n$ and $\rho:\mathbb R^n\rightarrow(0,+\infty)$ be a positive continuous function. The set $\omega$ is said to be thick with respect to the density $\rho$ if it is measurable and satisfies
$$\exists\gamma\in(0,1], \forall x\in\mathbb R^n,\quad \Leb(\omega\cap B(x,\rho(x)))\geq\gamma\Leb(B(x,\rho(x))),$$
where $\Leb$ stands for the Lebesgue measure on $\mathbb R^n$.
\end{dfn}

The authors of the same paper obtained a positive null-controllability result from control supports which are thick with respect to Lipschitz densities, for evolution equations enjoying smoothing properties in symmetric Gelfand-Shilov spaces $S^{1/2\gamma}_{1/2\gamma}(\mathbb R^n)$, with $1/2\le\gamma\le1$. This is Theorem 2.5 in \cite{MPS}, stated in Theorem \ref{12112020T2} in the present work. An example of such an equation is given by the equation (\hyperref[01202020E1]{$E_{s,l,l}$}) with $l\geq1$ according to Corollary \ref{02102020C1}, with $\gamma = \min(sl,1)$. As a consequence, we obtain the following

\begin{thm}\label{12112020T1} Let $l\geq1$ be a positive integer and $s>0$ be a positive real number satisfying $1/2<\min(sl,1)\le 1$. We consider a $1/2$-Lipschitz function $\rho:\mathbb R^n\rightarrow(0,+\infty)$, the space $\mathbb R^n$ being equipped with the canonical Euclidean norm, satisfying that there exist some positive constants $0\le\delta<\min(2sl-1,1)$ and $c,R>0$ such that 
$$\forall x\in\mathbb R^n,\quad c\le\rho(x)\le R\langle x\rangle^{\delta}.$$
When $\omega\subset\mathbb R^n$ is a thick set with respect to the density $\rho$, the evolution equation $\mathrm($\hyperref[01202020E1]{$E_{s,l,l}$}$\mathrm)$ is null-controllable from the control support $\omega$ in any positive time $T>0$.
\end{thm}

This result has already be proven in \cite{MPS} (Corollary 2.6) in the case where $l=1$ and $1/2<s\le1$. Moreover, since a thick subset of $\mathbb R^n$ is thick with respect to a constant density, Theorem \ref{12112020T1} is the exact generalization of Theorem \ref{09112020T1} when $k=m=l$.

Since the following inclusion holds
$$S^{\mu_{s,k,m}}_{\nu_{s,k,m}}(\mathbb R^n)\subset S^{\max(\mu_{s,k,m},\nu_{s,k,m})}_{\max(\mu_{s,k,m},\nu_{s,k,m})}(\mathbb R^n),$$
we get that generically, the semigroups generated by fractional anisotropic Shubin operators $H^s_{k,m}$ enjoy smoothing properties in symmetric Gelfand-Shilov spaces from Corollary \ref{02102020C1}. Therefore, by still using Theorem \ref{12112020T2}, one could state a result for any positive integers $k,m\geq1$, more general than Theorem \ref{12112020T1}. However, such a theorem would not be an extension of Theorem \ref{09112020T1} since it would require an assumption like $2s\min(k,m)>1$, instead of $2ms>1$. That is the reason why such a general null-controllability result is not stated in the present work. This forced symmetrization seems not to be the good approach to tackle null-controllability issues for the equation \eqref{01202020E1} in a general setting.

About that, it would be very interesting to unify Theorem \ref{09112020T1} and Theorem \ref{12112020T1}, that is, extending Theorem \ref{12112020T1} in the case where the positive integers $k,m\geq1$ can be different. We have already mentioned \cite{MR2984079} (Theorem 1.10) which states that when $k\geq2$, the equation (\hyperref[01202020E1]{$E_{k,1,1}$}) is null-controllable (since observable) in any positive time $T>0$ from cones of the form \eqref{20112020E2}. Yet, it follows from a straightforward computation that those cones are sets which are thick with respect to the densities $R\langle x\rangle$ with $R>1$. Theorem 1.10 in \cite{MR2984079} therefore turns out to be a first generalization of Theorem \ref{12112020T1} in this particular non-symmetric case.

Directly extending \cite{MPS} (Theorem 2.5), of which Theorem \ref{12112020T1} is an application (by using Corollary \ref{02102020C1}), for evolution equations enjoying smoothing properties in general Gelfand-Shilov spaces seems quite difficult. Indeed, the proof of this result is based on characterization of the symmetric Gelfand-Shilov spaces $S^{\mu}_{\mu}(\mathbb R^n)$ through the decomposition into the Hermite basis of $L^2(\mathbb R^n)$, and requires to use Bernstein type estimates for the Hermite functions, see \cite{MPS} (Section 3 and Theorem 5.2). As we have already mentioned, concerning the possibly non-symmetric Gelfand-Shilov spaces $S^{\mu}_{\nu}(\mathbb R^n)$ with $\mu/\nu\in\mathbb Q$, a similar characterization exists through the decomposition into the basis of eigenfunctions of the anisotropic Shubin operators $H_{k,m}$, as explained in Subsection \ref{GS} in the Appendix. However, Bernstein type estimates for the eigenfunctions of the operator $H_{k,m}$ have not been established in general yet, and seem to be more difficult to obtain than for the Hermite functions, for which we have an explicit formula. That is why the proof of \cite{MPS} (Theorem 2.5) therefore cannot be directly adapted. Nevertheless, we expect that the estimates given by Corollary \ref{02102020C1} will allow to obtain a generalization of Theorem \ref{12112020T1} when $k\geq 1$ and $m\geq1$ may be different.

Let us now consider the situation where $l=1$ and $s>1$. Theorem \ref{12112020T1} states in this case that the equation (\hyperref[01202020E1]{$E_{s,1,1}$}) is null-controllable in any positive time $T>0$ from control supports which are thick with respect to sublinear densities. It fact this result can be considerable sharpen, since we can prove that the equation (\hyperref[01202020E1]{$E_{s,1,1}$}) is null-controllable in any positive time $T>0$ from any measurable control support with positive Lebesgue measure.

\begin{thm}\label{09112020P1} Let $s>1$ be a positive real number. When $\omega\subset\mathbb R^n$ is any measurable set with positive Lebesgue measure, the evolution equation $\mathrm($\hyperref[01202020E1]{$E_{s,1,1}$}$\mathrm)$ is null-controllable from the control support $\omega$ in any positive time $T>0$.
\end{thm}

The statement and the proof of Theorem \ref{09112020P1} have been kindly communicated to the author by J. Martin. It is still an open question to know if Theorem \ref{09112020P1} still holds for any positive integer $l\geq1$ and any $sl>1$, and even more generally for all $k,m\geq1$ and $s>s_{cr}$, the critical exponent $s_{cr}$ being defined in \eqref{19112020E1}.

\section{Null-controllability of evolution equations associated with fractional anisotropic Shubin operators}\label{nullcont}

In this section, we explain how the quantitative smoothing properties given by Corollary \eqref{02102020C1} (which will be proven later in this work) and the Lebeau-Robbiano strategy allow to obtain the null-controllability results Theorem \ref{09112020T1} and Theorem \ref{12112020T1} presented in Subsection \ref{Null-Cont} and concerning the evolution equation
\begin{equation}\label{12112020E1}\tag{$E_{s,k,m}$}
\left\{\begin{aligned}
	& \partial_tf(t,x) + H^s_{k,m}f(t,x) = h(t,x)\mathbbm 1_{\omega}(x),\quad t>0,\ x\in\mathbb R^n, \\
	& f(0,\cdot) = f_0\in L^2(\mathbb R^n),
\end{aligned}\right.
\end{equation}
where $k,m\geq1$ are positive integers, $s>0$ is a positive real number and $H^s_{k,m}$ is the associated fractional anisotropic Shubin operator defined in \eqref{20112020E3} and equipped with the domain \eqref{20112020E4}. We will also present a proof of Theorem \ref{09112020P1} communicated to the author by J. Martin, which does not require any result proven in the present paper. Since the operator $H^s_{k,m}$ is selfadjoint, the Hilbert Uniqueness Method, see e.g. \cite{MR2302744} (Theorem 2.44), shows that the null-controllability of the equation \eqref{12112020E1} is equivalent to the observability of the adjoint system
\begin{equation}\label{24052018E2}\tag{${E_{s,k,m}}^*$}
\left\{\begin{array}{l}
	\partial_tg(t,x) + H_{k,m}^sg(t,x) = 0,\quad t>0,\ x\in\mathbb R^n, \\[5pt]
	g(0,\cdot) = g_0\in L^2(\mathbb R^n).
\end{array}\right.
\end{equation}

\begin{dfn} Given a positive time $T>0$ and a Borel set $\omega\subset\mathbb R^n$ with a positive Lebesgue measure, the equation \eqref{24052018E2} is said to be \textit{observable} from $\omega$ in time $T$ if there exists a positive constant $C(T,\omega)>0$ such that for all $g\in L^2(\mathbb R^n)$,
\begin{equation}\label{12112020E2}
	\big\Vert e^{-TH^s_{k,m}}g\big\Vert^2_{L^2(\mathbb R^n)}\le C(T,\omega)\int_0^T\big\Vert g(t,\cdot)\big\Vert^2_{L^2(\omega)}\,\mathrm dt.
\end{equation}
\end{dfn}

A classical method in observability theory is the Lebeau-Robbiano strategy, introduced in the work \cite{MR1312710} in order to study the null-controllability of the heat equation posed on bounded domains of $\mathbb R^n$. Essentially, this approach states that to obtain observability estimates like \eqref{12112020E2}, it sufficient to obtain a spectral inequality for the control support $\omega$ and a dissipation estimate for the semigroup solution of the equation \eqref{24052018E2}. In this work, we will use a recently revised version of this method (directly or through results also proven by the following result), due to K. Beauchard M. Egidi and K. Pravda-Starov in \cite{BEPS}, which is essentially a reformulation of a previous result due to L. Miller \cite{M} (involving a telescopic series), following the seminal ideas in \cite{MR1312710}.

\begin{thm}[Theorem 2.1 in \cite{BEPS}]\label{08122017T3} Let $\Omega\subset\mathbb R^n$ be an open set, $\omega\subset\Omega$ be a measurable subset, $(\pi_k)_{k\geq1}$ be a family of orthogonal projections defined on $L^2(\Omega)$ and $(e^{-tA})_{t\geq0}$ be a strongly continuous contraction semigroup on $L^2(\Omega)$. Assume that there exist some constants $c_1, c'_1, c_2, c'_2, a,b,t_0,m_1>0$ and $m_2\geq0$, with $a<b$, such that the following spectral inequality 
\begin{equation}\label{11112020E2}
	\forall g\in L^2(\mathbb R^n), \forall k\geq1,\quad \Vert\pi_kg\Vert_{L^2(\Omega)}\le c'_1e^{c_1k^a}\Vert\pi_kg\Vert_{L^2(\omega)},
\end{equation}
and the following dissipation estimate
\begin{equation}\label{11112020E3}
	\forall g\in L^2(\mathbb R^n), \forall k\geq1, \forall t\in(0,t_0),\quad \big\Vert(1-\pi_k)(e^{-tA}g)\big\Vert_{L^2(\Omega)}\le\frac{e^{-c_2t^{m_1}k^b}}{c'_2t^{m_2}}\Vert g\Vert_{L^2(\Omega)},
\end{equation}
hold. Then, there exists a positive constant $C>1$ such that the following observability estimate holds
$$\forall T>0, \forall g\in L^2(\mathbb R^n),\quad \big\Vert e^{-TA}g\big\Vert^2_{L^2(\mathbb R^n)}\le C\exp\bigg(\frac C{T^{\frac{am_1}{b-a}}}\bigg)\int_0^T\big\Vert e^{-tA}g\big\Vert^2_{L^2(\omega)}\,\mathrm dt.$$
\end{thm}

Notice that the spectral inequality \eqref{11112020E2} is an intrinsic geometric property of the control support $\omega$, while the dissipation estimate \eqref{11112020E3} only depends on the semigroup generated by the operator $A$. In our context, the latter will be a consequence of Corollary \ref{02102020C1}. Moreover, we will use spectral inequalities already existing in the literature for frequency cutoff projections and projections over the first modes of the Hermite basis of $L^2(\mathbb R^n)$.

\subsection{Anisotropic case}\label{ani} Let us begin by proving Theorem \ref{09112020T1}. Assume that $2sm>1$ and that the control support $\omega$ is a thick set. We consider the sequence $(\pi_k)_{k\geq1}$ of orthogonal frequency cutoff projections defined by
\begin{equation}\label{30112020E1}
	\pi_k : L^2(\mathbb R^n) \rightarrow\big\{g\in L^2(\mathbb R^n) : \Supp\widehat g\subset[-k,k]^n\big\},
\end{equation}
where $\widehat g\in L^2(\mathbb R^n)$ denotes the Fourier transform of the function $g\in L^2(\mathbb R^n)$. In order to apply Theorem \ref{08122017T3}, we need to use a spectral estimate of form \eqref{11112020E2} and to establish a dissipation estimate like \eqref{11112020E3}. Let us start with the latter. We deduce from Corollary \ref{02102020C1} that there exist some positive constants $c_1,c_2>0$ and $0<T<1$ such that for all $0<t<T$ and $g\in L^2(\mathbb R^n)$,
$$\big\Vert e^{c_1t\vert D_x\vert^{\frac1{\mu_{s,k,m}}}}(e^{-tH^s_{k,m}}g)\big\Vert_{L^2(\mathbb R^n)}\le\frac{c_2}{t^{\frac{n(k+m)}{2skm}}}\Vert g\Vert_{L^2(\mathbb R^n)},$$
the exponent $\mu_{s,k,m}>0$ being given by
$$\mu_{s,k,m} = \max\bigg(\frac1{2sm},\frac k{k+m}\bigg).$$
It  therefore follows from the definition \eqref{30112020E1} of the cutoff projections $\pi_k$ and Plancherel's theorem that for all $k\geq1$, $0<t<T$ and $g\in L^2(\mathbb R^n)$,
\begin{align}\label{11112020E4}
	\big\Vert(1-\pi_k)e^{-tH^s_{k,m}}g\big\Vert_{L^2(\mathbb R^n)} 
	& = \big\Vert(1-\pi_k) e^{-c_1t\vert D_x\vert^{\frac1{\mu_{s,k,m}}}} e^{c_1t\vert D_x\vert^{\frac1{\mu_{s,k,m}}}}(e^{-tH^s_{k,m}}g)\big\Vert_{L^2(\mathbb R^n)} \nonumber \\[5pt]
	& \le\frac{c_2}{t^{\frac{n(k+m)}{2skm}}} e^{-c_1tk^{\frac1{\mu_{s,k,m}}}}\Vert g\Vert_{L^2(\mathbb R^n)}. \nonumber
\end{align}
Notice that this is a dissipation estimate of the form \eqref{11112020E3} with $m_1 = 1$, $m_2 = \frac{n(k+m)}{2skm}$ and $b = 1/\mu_{s,k,m}$. On the other hand, concerning the spectral estimate, we use O. Kovrijkine's following result which is taken from the work \cite{MR1840110}:

\begin{thm}[Theorem 3 in \cite{MR1840110}]\label{Kovrij} There exists a universal positive constant $C_n>0$ depending only on the dimension $n\geq 1$ such that for all set $\omega\subset\mathbb R^n$ being $\gamma$-thick at scale $L>0$,
$$\forall k\geq1, \forall g\in L^2(\mathbb R^n),\quad\Vert \pi_kg\Vert_{L^2(\mathbb R^n)}\le\Big(\frac{C_n}{\gamma}\Big)^{C_n(1+Lk)}\Vert\pi_kg\Vert_{L^2(\omega)},$$
the orthogonal frequency cutoff projections $\pi_k$ being defined in \eqref{30112020E1}.
\end{thm}

In view of the definition of the orthogonal cutoff projections $\pi_k$, and the control support $\omega\subset\mathbb R^n$ being thick by assumption, we deduce from the above theorem that there exists a positive constant $c>0$ such that
\begin{equation}\label{11112020E5}
	\forall k\geq1,\forall g\in L^2(\mathbb R^n), \quad \Vert\pi_ku\Vert_{L^2(\mathbb R^n)}\le e^{ck}\Vert\pi_ku\Vert_{L^2(\omega)}.
\end{equation}
This is the spectral estimate \eqref{11112020E2} with $a=1$. Moreover, since we assumed $2sm>1$, we get that $1/\mu_{s,k,m}>1$, that is, $b>a$. We therefore deduce from Theorem \ref{08122017T3} that there exists a positive constant $C>1$ such that for all $T>0$ and $g\in L^2(\mathbb R^n)$, the following observability estimate holds
$$\big\Vert e^{-TH^s_{k,m}}g\big\Vert^2_{L^2(\mathbb R^n)}\le C\exp\bigg(\frac C{T^{\beta_{s,k,m}}}\bigg)\int_0^T\big\Vert e^{-tH^s_{k,m}}g\big\Vert^2_{L^2(\omega)}\,\mathrm dt,$$
with
$$\beta_{s,k,m} = \max\bigg(\frac1{2sm-1},\frac km\bigg).$$
This ends the proof of Theorem \ref{09112020T1}.

\subsection{Isotropic case} In this second subsection, we prove Theorem \ref{12112020T1}. We therefore assume that $k=m=l$. In fact, Theorem \ref{12112020T1} is a direct consequence of a result from the paper \cite{MPS} by J. Martin and K. Pravda-Starov. One of the purposes of this work, see Subsection 2.3, is to study the null-controllability of linear evolution equations posed on the whole space $\mathbb R^n$ and enjoying smoothing properties in symmetric Gelfand-Shilov spaces. More specifically, these authors consider strongly contraction semigroups $(e^{-tA})_{t\geq0}$ satisfying that there exist some positive constants $1/2<\gamma\le 1$, $C_{\gamma}>1$, $0<t_0<1$ and $m_1,m_2\in\mathbb R$ with $m_1>0,m_2\geq0$ such that for all $0<t<t_0$, $(\alpha,\beta)\in\mathbb N^{2n}$ and $g\in L^2(\mathbb R^n)$,
\begin{equation}\label{12112020E3}
	\big\Vert x^{\alpha}\partial^{\beta}_x(e^{-tA}g)\big\Vert_{L^2(\mathbb R^n)}\le\frac{C_{\gamma}^{1+\vert\alpha+\beta\vert}}{t^{m_1\vert\alpha+\beta\vert+m_2}}\ (\alpha!)^{\frac1{2\gamma}}\ (\beta!)^{\frac1{2\gamma}}\ \Vert g\Vert_{L^2(\mathbb R^n)}.
\end{equation}
By exploiting the Lebeau-Robbiano strategy, and more precisely Theorem \ref{08122017T3} applied with spectral inequalities for finite combinations of Hermite functions obtained in the same work \cite{MPS} (Theorem 2.1), J. Martin and K. Pravda-Starov established the following positive null-controllability result:

\begin{thm}[Theorem 2.5 in \cite{MPS}]\label{12112020T2} Let $A$ be a closed operator on $L^2(\mathbb R^n)$ which is the infinitesimal generator of a strongly continuous contraction semigroup $(e^{-tA})_{t\geq0}$ on $L^2(\mathbb R^n)$ that satisfies the quantitative smoothing estimates \eqref{12112020E3} for some $1/2<\gamma\le1$. We consider a $1/2$-Lipschitz function $\rho:\mathbb R^n\rightarrow(0,+\infty)$, the space $\mathbb R^n$ being equipped with the canonical Euclidean norm, satisfying that there exist some positive constants $0\le\delta<2\gamma-1$, and $c_1,c_2>0$ such that 
$$\forall x\in\mathbb R^n,\quad c_1\le\rho(x)\le c_2\langle x\rangle^{\delta}.$$
If $\omega\subset\mathbb R^n$ is a measurable set which is thick with respect to the density $\rho$, then the evolution equation associated with the $L^2(\mathbb R^n)$-adjoint $A^*$ of the operator $A$
$$\left\{\begin{aligned}
	& \partial_tf(t,x) + A^*f(t,x) = h(t,x)\mathbbm 1_{\omega}(x),\quad t>0,\ x\in\mathbb R^n, \\
	& f(0,\cdot) = f_0\in L^2(\mathbb R^n),
\end{aligned}\right.$$
is null-controllable from the set $\omega$ in any positive time $T>0$.
\end{thm}

Let us recall that the operator $H^s_{l,l}$ we are considering in this subsection is selfadjoint. Moreover, we deduce from \eqref{20112020E5} that there exists a positive constant $C>0$ such that for all $0<t<T$, $(\alpha,\beta)\in\mathbb N^{2n}$ and $g\in L^2(\mathbb R^n)$,
$$\big\Vert x^{\alpha}\partial^{\beta}_x(e^{-tH^s_{l,l}}g)\big\Vert_{L^2(\mathbb R^n)}\le\frac{C^{\vert\alpha+\beta\vert}}{t^{\frac{\vert\alpha+\beta\vert}{\min(2sl,2)}+\frac n{sl}}}\ (\alpha!)^{\frac1{\min(2sl,2)}}\ (\beta!)^{\frac1{\min(2sl,2)}}\ \Vert g\Vert_{L^2(\mathbb R^n)}.$$
The proof of Theorem \ref{12112020T1} is therefore ended after using Theorem \ref{12112020T2}.

\subsection{Fractional harmonic oscillator} To end this section, let us present the proof of Proposition \ref{09112020P1}. As announced, the following proof was communicated to the author by J. Martin. It does not directly use resuts obtained in the present work but is based on Theorem \ref{08122017T3} again. We assume that $s>1$ and $k=m=1$, that is, we consider the evolution equation associated with large fractional powers of the standard harmonic oscillator. Let us consider this time a measurable set $\omega\subset\mathbb R^n$ with a positive Lebesgue measure and the sequence $(p_k)_{k\geq1}$ of orthogonal cutoff projections with respect to the Hermite basis of $L^2(\mathbb R^n)$ defined by
\begin{equation}\label{12112020E5}
	p_k : L^2(\mathbb R^n) \rightarrow\Vect_{\mathbb C}\{\Phi_{\alpha}\}_{\vert\alpha\vert\le k},
\end{equation}
where $(\Phi_{\alpha})_{\alpha\in\mathbb N^n}$ denotes the Hermite basis of $L^2(\mathbb R^n)$. On the one hand, since the eigenvalues of the harmonic oscillator $H_{1,1}$ associated with the eigenfunction $\Phi_{\alpha}$ is given by $2\vert\alpha\vert+n$, we get from Parseval's formula that for all $t\geq0$,
\begin{align}\label{12112020E4}
	\big\Vert(1-p_k)e^{-tH^s_{1,1}}g\big\Vert^2_{L^2(\mathbb R^n)} 
	& = \sum_{\vert\alpha\vert\geq k+1}\big\vert\langle g,\Phi_{\alpha}\rangle_{L^2(\mathbb R^n)}\big\vert^2e^{-2t(2\vert\alpha\vert+n)^s} \\[5pt]
	& \le e^{-2t(2(k+1)+n)^s}\Vert g\Vert^2_{L^2(\mathbb R^n)}. \nonumber
\end{align}
On the other hand, we use the following spectral inequalities for finite combinaisons of Hermite functions proven by K. Beauchard, P. Jaming and K. Pravda-Starov:

\begin{thm}[Theorem 2.1 in \cite{BJKPS}] If $\omega\subset\mathbb R^n$ is a measurable set with a positive Lebesgue measure, then there exists a positive constant $C=C(\omega)>1$ such that
$$\forall k\geq1, \forall g\in L^2(\mathbb R^n),\quad\Vert p_k g\Vert_{L^2(\mathbb R^n)}\le Ce^{\frac12k\log(k+1) + Ck}\Vert p_kg\Vert_{L^2(\omega)},$$
the orthogonal cutoff projections $p_k$ being the ones defined in \eqref{12112020E5}.
\end{thm}

Actually, this spectral inequality was originally stated for non-empty open sets $\omega\subset\mathbb R^n$ in \cite{BJKPS} and then extended to Borel sets with positive measures in the work \cite{HWW} (Lemma 3.2) by borrowing some ideas from the proof of \cite{BJKPS} (Theorem 2.1). Since $s>1$ by assumption, we can consider $1<s'<s$ a positive real number. It follows from the above theorem that there exists a positive constant $c>0$ such that 
$$\forall k\geq1, \forall g\in L^2(\mathbb R^n),\quad\Vert p_k g\Vert_{L^2(\mathbb R^n)}\le ce^{ck^{s'}}\Vert p_kg\Vert_{L^2(\omega)}.$$
This spectral inequality and the dissipation estimate \eqref{12112020E4} end the proof of Proposition \ref{09112020P1} thanks to Theorem \ref{08122017T3} (recall that $s'<s$).

\section{Agmon estimates for anisotropic Shubin operators}\label{agmon}

This section is devoted to the proof of Theorem \ref{19022020T1}. In the following, we will not use any results existing in the literature concerning the Schwartz regularity of the eigenfunctions of anisotropic Shubin operators, since we want to recover the Gelfand-Shilov regularity of those eigenfunctions, with new precise estimates of the associated seminorms. Let $0\le\sigma\le1$ be a non-negative real number, $k,m\geq1$ be positive integers and $H_{k,m}$ be the associated anisotropic Shubin operator defined in \eqref{19022020E1} and equipped with domain \eqref{25092020E1}. It is sufficient to prove that there exist some positive constants $c_1,c_2>0$ and $T>0$ such that for all eigenfunction $\psi\in L^2(\mathbb R^n)$ of the operator $H_{k,m}$ and $0\le t\le T$,
\begin{equation}\label{21022020E3}
	\big\Vert e^{c_1t\langle x\rangle^{\sigma(1+\frac km)}}\psi\big\Vert_{L^2(\mathbb R^n)} \le c_2e^{c_2t\lambda^{\sigma(\frac1{2k}+\frac1{2m})}}\Vert\psi\Vert_{L^2(\mathbb R^n)},
\end{equation}
with $\lambda>0$ the eigenvalue associated with the eigenfunction $\psi$. Indeed, notice that $\psi\in L^2(\mathbb R^n)$ is an eigenfunction of the operator $H_{k,m}$ if and only if its Fourier transform $\widehat{\psi}$ is an eigenfunction of the operator $H_{m,k}$ associated with the same eigenvalue.  As a consequence, once \eqref{21022020E3} is established, we deduce by exchanging the roles of the integers $k$ and $m$ that for all eigenfunction $\psi\in L^2(\mathbb R^n)$ of the operator $H_{k,m}$ and $0\le t\le T$,
$$\big\Vert e^{c_1t\langle x\rangle^{\sigma(1+\frac mk)}}\widehat{\psi}\big\Vert_{L^2(\mathbb R^n)} \le c_2e^{c_2t\lambda^{\sigma(\frac1{2m}+\frac1{2k})}}\Vert\widehat{\psi}\Vert_{L^2(\mathbb R^n)},$$
with $\lambda>0$ the eigenvalue associated with the eigenfunction $\psi$. Plancherel's theorem then implies that 
$$\big\Vert e^{c_1t\langle D_x\rangle^{\sigma(1+\frac mk)}}\psi\big\Vert_{L^2(\mathbb R^n)} \le c_2e^{c_2t\lambda^{\sigma(\frac1{2k}+\frac1{2m})}}\Vert\psi\Vert_{L^2(\mathbb R^n)},$$
which ends the proof of Theorem \ref{19022020T1}. We therefore focus on proving the estimate \eqref{21022020E3}. Let $\psi\in L^2(\mathbb R^n)$ be an eigenfunction of the operator $H_{k,m}$ associated with the eigenvalue $\lambda>0$. We consider the smooth function $\phi\in C^{\infty}(\mathbb R^n,\mathbb R)$ defined for all $x\in\mathbb R^n$ by
\begin{equation}\label{05032020E1}
	\phi(x) = \langle x\rangle^{\sigma(1+\frac km)}.
\end{equation}
In the following, we will need to deal with a compactly supported approximation of the function $\psi$. To that end, let us consider a cut-off odd function $\chi\in C^{\infty}_0(\mathbb R,\mathbb R)$ satisfying that $\chi(x) = x$ for all $0\le x\le1$, $\chi(x) = 0$ when $x\geq2$ and $\chi(x)\geq0$ for all $x\geq0$. For all $\varepsilon>0$, we consider the compactly supported functions $\chi_{\varepsilon}$ and $\phi_{\varepsilon}$ respectively defined for all $x\in\mathbb R^n$ by
\begin{equation}\label{03022020E1}
	\chi_{\varepsilon}(x) = \frac1{\varepsilon}\chi(\varepsilon x)\quad \text{and}\quad \phi_{\varepsilon}(x) = (\chi_{\varepsilon}\circ\phi)(x).
\end{equation}
Notice that by construction, the family $(\chi_{\varepsilon})_{\varepsilon>0}$ is an approximation of the identity function. We also need to deal with a Schwartz approximation of the eigenfunction $\psi$. We therefore consider $(\psi_j)_j$ a sequence in $\mathscr S(\mathbb R^n)$, given by Proposition \ref{07102020P1} in Appendix, and satisfying
\begin{equation}\label{21022020E6}
	\lim_{j\rightarrow+\infty}\psi_j = \psi\quad\text{and}\quad \lim_{j\rightarrow+\infty}H_{k,m}\psi_j = H_{k,m}\psi\quad \text{in $L^2(\mathbb R^n)$}.
\end{equation}
We can now tackle the proof of the estimate \eqref{21022020E3}. Our approach is based on a classical Agmon strategy. The first step consists in noticing that for all $j\geq0$, $\varepsilon>0$ and $t\geq0$, the term $\langle H_{k,m}\psi_j,e^{2t\phi_{\varepsilon}}\psi_j\rangle_{L^2}$ can be written in the two following ways
$$\big\langle H_{k,m}\psi_j,e^{2t\phi_{\varepsilon}}\psi_j\big\rangle_{L^2(\mathbb R^n)} = \lambda\big\Vert e^{t\phi_{\varepsilon}}\psi_j\big\Vert^2_{L^2(\mathbb R^n)} 
+  \big\langle H_{k,m}\psi_j-\lambda\psi_j,e^{2t\phi_{\varepsilon}}\psi_j\big\rangle_{L^2(\mathbb R^n)},$$
and
$$\big\langle H_{k,m}\psi_j,e^{2t\phi_{\varepsilon}}\psi_j\big\rangle_{L^2(\mathbb R^n)} = \big\langle(-\Delta)^m\psi_j,e^{2t\phi_{\varepsilon}}\psi_j\big\rangle_{L^2(\mathbb R^n)} + \big\langle\vert x\vert^{2k}\psi_j,e^{2t\phi_{\varepsilon}}\psi_j\big\rangle_{L^2(\mathbb R^n)}.$$
The second step consists in controlling the term involving $(-\Delta)^m$. After making the change of function $v_j = e^{t\phi_{\varepsilon}}\psi_j\in\mathscr S(\mathbb R^n)$, this term writes in the most useful following form 
\begin{equation}\label{21022020E7}
	\langle(-\Delta)^m\psi_j,e^{2t\phi_{\varepsilon}}\psi_j\rangle_{L^2(\mathbb R^n)} = \big\langle e^{t\phi_{\varepsilon}}(-\Delta)^m(e^{-t\phi_{\varepsilon}}v_j),v_j\big\rangle_{L^2(\mathbb R^n)}.
\end{equation}
In order to manage it, we use the estimate given by the following proposition which provides a G{\aa}rding type inequality and whose proof is postponed in Section \ref{secgard}:

\begin{prop}\label{19022020P1} There exists a positive constant $c_0>0$ depending on the function $\phi$ such that for all $0<\varepsilon\le1$, $0\le t\le1$ and $v\in\mathscr S(\mathbb R^n)$,
$$\big\langle e^{t\phi_{\varepsilon}}(-\Delta)^m(e^{-t\phi_{\varepsilon}}v),v\big\rangle_{L^2(\mathbb R^n)}
+c_0\Big(\Vert v\Vert^2_{L^2(\mathbb R^n)} + t\big\Vert\langle x\rangle^{\sigma k}v\big\Vert^2_{L^2(\mathbb R^n)}\Big)\geq0.$$
\end{prop}

We deduce from \eqref{21022020E7}, the above proposition and Cauchy-Schwarz' inequality that for all $j\geq0$, $0<\varepsilon\le1$ and $0\le t\le1$,
\begin{multline*}
	\big\langle\vert x\vert^{2k}\psi_j,e^{2t\phi_{\varepsilon}}\psi_j\big\rangle_{L^2(\mathbb R^n)} - c_0\big\Vert e^{t\phi_{\varepsilon}}\psi_j\big\Vert^2_{L^2(\mathbb R^n)} - c_0t\big\Vert\langle x\rangle^{\sigma k}e^{t\phi_{\varepsilon}}\psi_j\big\Vert^2_{L^2(\mathbb R^n)} \\[5pt]
	\le\lambda\big\Vert e^{t\phi_{\varepsilon}}\psi_j\big\Vert^2_{L^2(\mathbb R^n)} + \big\Vert H_{k,m}\psi_j-\lambda\psi_j\big\Vert_{L^2(\mathbb R^n)}\big\Vert e^{2t\phi_{\varepsilon}}\psi_j\big\Vert_{L^2(\mathbb R^n)}.
\end{multline*}
This estimate also takes the following integral form
$$\int_{\mathbb R^n}\big(\vert x\vert^{2k} - c_0t\langle x\rangle^{2\sigma k} - c_0 - \lambda\big)\ e^{2t\phi_{\varepsilon}(x)}\vert\psi_j(x)\vert^2\,\mathrm dx
\le\big\Vert H_{k,m}\psi_j-\lambda\psi_j\big\Vert_{L^2(\mathbb R^n)}\big\Vert e^{2t\phi_{\varepsilon}}\psi_j\big\Vert_{L^2(\mathbb R^n)}.$$
The third step of the proof consists in managing the above integral. To that end, we will distinguish the two regions in $\mathbb R^n$ where $\vert x\vert^{2k}>M\lambda$ and $\vert x\vert^{2k}\le M\lambda$ respectively, with $M\gg1$ a large constant independent of the Schwartz functions $\psi_j$ and the eigenvalue $\lambda$ whose value will be chosen later, by writing 
\begin{multline}\label{21022020E2}
	\int_{\vert x\vert^{2k}>M\lambda}\big(\vert x\vert^{2k} - c_0t\langle x\rangle^{2\sigma k} - c_0 - \lambda\big)\ e^{2t\phi_{\varepsilon}(x)}\vert\psi_j(x)\vert^2\,\mathrm dx \\
	\le\big\Vert H_{k,m}\psi_j-\lambda\psi_j\big\Vert_{L^2(\mathbb R^n)}\big\Vert e^{2t\phi_{\varepsilon}}\psi_j\big\Vert_{L^2(\mathbb R^n)} \\[5pt]
	+ \int_{\vert x\vert^{2k}\le M\lambda}\big(\lambda + c_0t\langle x\rangle^{2\sigma k} + c_0 - \vert x\vert^{2k}\big)\ e^{2t\phi_{\varepsilon}(x)}\vert\psi_j(x)\vert^2\,\mathrm dx.
\end{multline}
On the one hand, since $0\le \sigma\le1$, there exist some positive constants $0<t_0\le1$ and $c_1>0$ such that
$$\forall t\in[0,t_0], \forall x\in\mathbb R^n,\quad c_0t\langle x\rangle^{2\sigma k} + c_0 - \vert x\vert^{2k}\le c_1,$$
and we obtain the following upper bound
\begin{multline}\label{21022020E1}
	\int_{\vert x\vert^{2k}\le M\lambda}\big(\lambda + c_0t\langle x\rangle^{2\sigma k} + c_0 - \vert x\vert^{2k}\big)\ e^{2t\phi_{\varepsilon}(x)}\vert\psi_j(x)\vert^2\,\mathrm dx. \\[5pt]
	\le(\lambda + c_1)\int_{\vert x\vert^{2k}\le M\lambda} e^{2t\phi_{\varepsilon}(x)}\vert\psi_j(x)\vert^2\,\mathrm dx.
\end{multline}
On the other hand, since $0\le\sigma\le1$ anew, we notice that there exist other positive constants $0<t_1\le t_0$, $r_0\gg1$ and $c_2>0$ such that for all $0\le t\le t_1$ and $x\in\mathbb R^n$ satisfying $\vert x\vert^{2k}\geq r_0$,
$$\vert x\vert^{2k} - c_0t\langle x\rangle^{2\sigma k} - c_0\geq c_2\vert x\vert^{2k}.$$
As a consequence, if the large positive constant $M\gg1$ satisfies $M\lambda_0\geq r_0$, with $\lambda_0>0$ the smallest eigenvalue of the operator $H_{k,m}$, we have $M\lambda\geq M\lambda_0\geq r_0$, and therefore,
\begin{align}\label{19022020E3}
	&\ \int_{\vert x\vert^{2k}>M\lambda}\big(\vert x\vert^{2k} - c_0t\langle x\rangle^{2\sigma k} - c_0 - \lambda\big)\ e^{2t\phi_{\varepsilon}(x)}\vert\psi_j(x)\vert^2\,\mathrm dx \\[5pt]
	\geq &\ \int_{\vert x\vert^{2k}>M\lambda}\big(c_2\vert x\vert^{2k} - \lambda\big)\ e^{2t\phi_{\varepsilon}(x)}\vert\psi_j(x)\vert^2\,\mathrm dx \nonumber \\[5pt]
	\geq &\ (c_2M - 1)\lambda\int_{\vert x\vert^{2k}>M\lambda}e^{2t\phi_{\varepsilon}(x)}\vert\psi_j(x)\vert^2\,\mathrm dx. \nonumber
\end{align}
The constant $M\gg1$ can be chosen large enough so that $c_3>0$, where we set $c_3 = c_2M-1$, and its value is now fixed. We deduce from \eqref{21022020E2}, \eqref{21022020E1} and \eqref{19022020E3} that for all $j\geq0$, $0<\varepsilon\le1$ and $0\le t\le t_1$,
\begin{multline*}
	\int_{\vert x\vert^{2k}>M\lambda}e^{2t\phi_{\varepsilon}(x)}\vert\psi_j(x)\vert^2\,\mathrm dx
	\le\frac1{c_3\lambda}\big\Vert H_{k,m}\psi_j-\lambda\psi_j\big\Vert_{L^2(\mathbb R^n)}\big\Vert e^{2t\phi_{\varepsilon}}\psi_j\big\Vert_{L^2(\mathbb R^n)} \\[5pt]
	+ \frac1{c_3}\Big(1 + \frac{c_1}{\lambda}\Big)\int_{\vert x\vert^{2k}\le M\lambda} e^{2t\phi_{\varepsilon}(x)}\vert\psi_j(x)\vert^2\,\mathrm dx.
\end{multline*}
Given that
$$\big\Vert e^{t\phi_{\varepsilon}}\psi_j\big\Vert^2_{L^2(\mathbb R^n)} = \int_{\vert x\vert^{2k}>M\lambda}e^{2t\phi_{\varepsilon}(x)}\vert\psi_j(x)\vert^2\,\mathrm dx
+ \int_{\vert x\vert^{2k}\le M\lambda}e^{2t\phi_{\varepsilon}(x)}\vert\psi_j(x)\vert^2\,\mathrm dx,$$
the above estimate implies that for all $j\geq0$, $0<\varepsilon\le1$ and $0\le t\le t_1$,
\begin{multline}\label{21022020E5}
	\big\Vert e^{t\phi_{\varepsilon}}\psi_j\big\Vert^2_{L^2(\mathbb R^n)}
	\le\frac1{c_3\lambda}\big\Vert H_{k,m}\psi_j-\lambda\psi_j\big\Vert_{L^2(\mathbb R^n)}\big\Vert e^{2t\phi_{\varepsilon}}\psi_j\big\Vert_{L^2(\mathbb R^n)} \\[5pt]
	+ \Big(\frac1{c_3}\Big(1 + \frac{c_1}{\lambda}\Big)+1\Big)\int_{\vert x\vert^{2k}\le M\lambda} e^{2t\phi_{\varepsilon}(x)}\vert\psi_j(x)\vert^2\,\mathrm dx.
\end{multline}
We need to be more precise concerning the second term of the right-hand side of the above estimate. First, it follows from the definition of the function $\chi$ and the definitions \eqref{03022020E1} of the functions $\chi_{\varepsilon}$ and $\phi_{\varepsilon}$ that there exists a positive constant $c_4>0$ such that for all $\varepsilon>0$ and $x\in\mathbb R^n$,
$$\phi_{\varepsilon}(x) = \chi_{\varepsilon}(\phi(x)) = \frac1{\varepsilon}\chi(\varepsilon\phi(x))\le c_4\phi(x).$$
Moreover, Minkowski's inequality and the classical inequality 
$$\forall a,b\geq0,\forall q>0,\quad (a+b)^q\le 2^{(q-1)_+}(a^q+b^q)\quad\text{with}\quad (q-1)_+ = \max(q-1,0),$$
imply that for all $x\in\mathbb R^n$,
$$\phi(x) = \langle x\rangle^{\sigma(1+\frac km)}\le(1+\vert x\vert)^{\sigma(1+\frac km)} \le2^{(\sigma(1+\frac km)-1)_+}\big(1+\vert x\vert^{\sigma(1+\frac km)}\big).$$
We deduce that for all $0<\varepsilon\le1$ and $x\in\mathbb R^n$ satisfying $\vert x\vert^{2k}\le M\lambda$,
$$\phi_{\varepsilon}(x)\le 2^{(\sigma(1+\frac km)-1)_+}c_4\big(1+(M\lambda)^{\frac{\sigma}{2k}+\frac{\sigma}{2m}}\big).$$
In addition, still denoting by $\lambda_0>0$ the smallest eigenvalue of the operator $H_{k,m}$, the following estimate holds
$$\frac1{c_3}\Big(1 + \frac{c_1}{\lambda}\Big)\le\frac1{c_3}\Big(1+\frac{c_1}{\lambda_0}\Big).$$
We therefore deduce that there exist some positive constants $c_5,c_6>0$ such that for all $j\geq0$, $0<\varepsilon\le1$ and $0\le t\le t_1$,
$$\Big(\frac1{c_3}\Big(1 + \frac{c_1}{\lambda}\Big)+1\Big)\int_{\vert x\vert^{2k}\le M\lambda} e^{2t\phi_{\varepsilon}(x)}\vert\psi_j(x)\vert^2\,\mathrm dx
\le c_5e^{c_6t\lambda^{\sigma(\frac1{2k}+\frac1{2m})}} \Vert\psi_j\Vert^2_{L^2(\mathbb R^n)},$$
and then, using anew that $\lambda\geq\lambda_0>0$, we get
$$\big\Vert e^{t\phi_{\varepsilon}}\psi_j\big\Vert^2_{L^2(\mathbb R^n)}\le\frac1{c_3\lambda_0}\big\Vert H_{k,m}\psi_j-\lambda\psi_j\big\Vert_{L^2(\mathbb R^n)}\big\Vert e^{2t\phi_{\varepsilon}}\psi_j\big\Vert_{L^2(\mathbb R^n)} + c_5e^{c_6t\lambda^{\sigma(\frac1{2k}+\frac1{2m})}} \Vert\psi_j\Vert^2_{L^2(\mathbb R^n)}.$$
By passing to the limit $j\rightarrow+\infty$ in this estimate while using \eqref{21022020E6}, and recalling that $H_{k,m}\psi = \lambda\psi$, we obtain the following inequality for all $0<\varepsilon\le1$ and $0\le t\le t_1$,
$$\big\Vert e^{t\phi_{\varepsilon}}\psi\big\Vert^2_{L^2(\mathbb R^n)}\le c_5e^{c_6t\lambda^{\sigma(\frac1{2k}+\frac1{2m})}} \Vert\psi\Vert^2_{L^2(\mathbb R^n)}.$$
The estimate \eqref{21022020E3} is then a consequence of Fatou's lemma, since we get that for all $0\le t\le t_1$,
\begin{align*}
	\big\Vert e^{t\phi}\psi\big\Vert^2_{L^2(\mathbb R^n)} = \big\Vert\liminf_{\varepsilon\rightarrow0} e^{t\phi_{\varepsilon}}\psi\big\Vert^2_{L^2(\mathbb R^n)}
	& \le \liminf_{\varepsilon\rightarrow0} \big\Vert e^{t\phi_{\varepsilon}}\psi\big\Vert^2_{L^2(\mathbb R^n)} \\[5pt]
	& \le c_5e^{c_6t\lambda^{\sigma(\frac1{2k}+\frac1{2m})}} \Vert\psi\Vert^2_{L^2(\mathbb R^n)}.
\end{align*}
Notice that both constants $c_5>0$ and $c_6>0$ do not depend on the eigenfunction $\psi\in L^2(\mathbb R^n)$ nor on the associated eigenvalue $\lambda>0$. This ends the proof of the estimate \eqref{21022020E3}.

\section{Smoothing properties of the associated semigroups}\label{smoothing}

Let $k,m\geq1$ be some positive integers, $s>0$ be a positive real number and $H^s_{k,m}$ be the associated fractional anisotropic Shubin operator defined in \eqref{20112020E3} and equipped with the domain \eqref{20112020E4}. The aim of this section is to study the smoothing properties enjoyed by the evolution operators generated by the operator $H^s_{k,m}$ on $L^2(\mathbb R^n)$.

\subsection{The general case}\label{general} This subsection is devoted to derive Corollary \ref{02102020C1} from Theorem \ref{19022020T1}. Let $(\psi_j)_j$ be a Hilbert basis of $L^2(\mathbb R^n)$ composed of eigenfunctions of the operator $H_{k,m}$ and $\lambda_j>0$ the eigenvalue associated with the eigenfunction $\psi_j$ for all $j\geq0$. Moreover, let $c_1,c_2>0$ and $T>0$ be the positive constants given by Theorem \ref{19022020T1}. We first prove that there exists a positive constant $c>0$ such that for all $0<t<T$ and $g\in L^2(\mathbb R^n)$,
\begin{equation}\label{02102020E3}
	\big\Vert e^{c_1t\langle x\rangle^{\frac1{\nu_{s,k,m}}}}(e^{-(1+c_2)tH^s_{k,m}}g)\big\Vert_{L^2(\mathbb R^n)} \le\frac c{t^{\frac{n(k+m)}{2skm}}}\Vert g\Vert_{L^2(\mathbb R^n)},
\end{equation}
the regularity exponent $\nu_{s,k,m}>0$ being given by
$$\nu_{s,k,m} = \max\bigg(\frac1{2sk},\frac m{k+m}\bigg).$$
In the following, the operator $H^s_{k,m}$ and the constant $\nu_{s,k,m}$ will simply be denoted $H^s$ and $\nu_s$ respectively in order to alleviate the writing. The strategy to obtain this estimate is to prove that there exists a positive constant $c>0$ such that for all $0<t<T$ and $g\in L^2(\mathbb R^n)$,
\begin{equation}\label{08102020E1}
	\sum_{j=0}^{+\infty}\big\Vert\langle e^{-(1+c_2)tH^s}g,\psi_j\rangle_{L^2(\mathbb R^n)}e^{c_1t\langle x\rangle^{\frac1{\nu_s}}}\psi_j\big\Vert_{L^2(\mathbb R^n)}\le\frac c{t^{\frac{n(k+m)}{2skm}}}\Vert g\Vert_{L^2(\mathbb R^n)}.
\end{equation}
Since the normed vector space $L^2(\mathbb R^n)$ is a Banach space, the above inequality implies that for all $0<t<T$ and $g\in L^2(\mathbb R^n)$,
$$\sum_{j=0}^{+\infty}\langle e^{-(1+c_2)tH^s}g,\psi_j\rangle_{L^2(\mathbb R^n)}e^{c_1t\langle x\rangle^{\frac1{\nu_s}}}\psi_j = e^{c_1t\langle x\rangle^{\frac1{\nu_s}}}(e^{-(1+c_2)tH^s}g)\in L^2(\mathbb R^n),$$
and also that the estimate \eqref{02102020E3} holds. We therefore focus on obtaining \eqref{08102020E1}. To that end, we begin by noticing that the exponent $1/\nu_s$ can be written in the following way
\begin{equation}\label{07122022E1}
	\frac1{\nu_s} = \sigma_s\bigg(1+\frac km\bigg)\quad\text{with}\quad\sigma_s = \min\bigg(\frac{2skm}{k+m},1\bigg)\in[0,1].
\end{equation}
By using that
$$\sigma_s\bigg(\frac1{2k}+\frac1{2m}\bigg) = \min\bigg(s,\frac1{2k}+\frac1{2m}\bigg)\le s,$$
and the fact that $\lim_{+\infty}\lambda_j = +\infty$, we deduce that
$$\exists j_0\geq1, \forall j\geq j_0,\quad\lambda_j^s\geq\lambda_j^{\sigma_s(\frac1{2k}+\frac1{2m})}.$$
Cauchy-Schwarz' inequality and Theorem \ref{19022020T1} then imply that for all $0<t<T$ and $g\in L^2(\mathbb R^n)$,
\begin{align*}
	&\ \sum_{j=0}^{+\infty}\big\Vert\langle e^{-(1+c_2)tH^s}g,\psi_j\rangle_{L^2(\mathbb R^n)}e^{c_1t\langle x\rangle^{\frac1{\nu_s}}}\psi_j\big\Vert_{L^2(\mathbb R^n)} \\[5pt]
	\le &\ \bigg(\sum_{j=0}^{+\infty}e^{-(1+c_2)t\lambda^s_j}\big\Vert e^{c_1t\langle x\rangle^{\frac1{\nu_s}}}\psi_j\big\Vert_{L^2(\mathbb R^n)}\bigg)\Vert g\Vert_{L^2(\mathbb R^n)} \\[5pt]
	\le &\ c_2\bigg(\sum_{j=0}^{+\infty}e^{-(1+c_2)t\lambda^s_j}e^{c_2t\lambda_j^{\sigma_s(\frac1{2k}+\frac1{2m})}}\bigg)\Vert g\Vert_{L^2(\mathbb R^n)}
	\le c_0c_2\bigg(\sum_{j=0}^{+\infty}e^{-t\lambda^s_j}\bigg)\Vert g\Vert_{L^2(\mathbb R^n)},
\end{align*}
where we set
$$c_0 = \max_{0\le j\le j_0-1}\sup_{0\le t\le T}e^{-c_2t(\lambda_j^s-\lambda_j^{\sigma_s(\frac1{2k}+\frac1{2m})})}>0.$$
Moreover, the result \cite{BBR} (Chapter 2, Corollary 3.1) implies that the asymptotic behavior of the eigenvalues $\lambda_j$ is the following
$$\lambda_j\underset{j\rightarrow+\infty}{\sim} c_{k,m}j^{\frac{2km}{n(k+m)}},$$
where $c_{k,m}>0$ is a positive constant only depending on the positive integers $k,m\geq1$. Thus, there exists a positive constant ${c_0}'>0$ such that for all $0<t<T$,
$$\sum_{j=0}^{+\infty}e^{-t\lambda^s_j}\le\sum_{j=0}^{+\infty}e^{-{c_0}'tj^{\frac{2skm}{n(k+m)}}}\le\int_{-1}^{+\infty}e^{-{c_0}'tx^{\frac{2skm}{n(k+m)}}}\,\mathrm dx.$$
We deduce that there exists another positive constant $c>0$ such that for all $0<t<T$ and $g\in L^2(\mathbb R^n)$,
$$\sum_{j=0}^{+\infty}\big\Vert\langle e^{-(1+c_2)tH^s}g,\psi_j\rangle_{L^2(\mathbb R^n)}e^{c_1t\langle x\rangle^{\frac1{\nu_s}}}\psi_j\big\Vert_{L^2(\mathbb R^n)}\le\frac c{t^{\frac{n(k+m)}{2skm}}}\Vert g\Vert_{L^2(\mathbb R^n)}.$$
This ends the proof of the estimate \eqref{08102020E1} and, therefore, the one of \eqref{02102020E3}. Proceeding in the very same way, we get that for all $0<t<T$ and $g\in L^2(\mathbb R^n)$,
$$\big\Vert e^{c_1t\langle D_x\rangle^{\frac1{\mu_{s,k,m}}}}(e^{-(1+c_2)tH^s_{k,m}}g)\big\Vert_{L^2(\mathbb R^n)} \le\frac c{t^{\frac{n(k+m)}{2skm}}}\Vert g\Vert_{L^2(\mathbb R^n)},$$
the regularity exponent $\mu_{s,k,m}>0$ being this time given by
$$\mu_{s,k,m} = \max\bigg(\frac1{2sm},\frac k{k+m}\bigg).$$
Indeed, the proof goes the same way as before, since we have
$$\frac1{\mu_{s,k,m}} = \min\bigg(2sm,1+\frac mk\bigg) = \sigma_s\bigg(1+\frac mk\bigg),$$
where $\sigma_s\in[0,1]$ is the same as in \eqref{07122022E1}. The proof of Corollary \ref{02102020C1} is now ended.

\subsection{The non-fractional case} To end this section, we improve the Gelfand-Shilov estimates given by Corollary \ref{02102020C1} in the non-fractional case, that is, when $s=1$, by proving Theorem \ref{14102020E1}. In the following, we will use steps or results already present in the proof of Theorem \ref{19022020T1} in Section \ref{agmon}. First, as we have already noticed, it is sufficient to obtain the existence of some positive constants $c_1>0$, $c_2>0$ and $T>0$ such that for all $0\le t\le T$ and $g\in L^2(\mathbb R^n)$,
\begin{equation}\label{05122019E3}
	\big\Vert e^{c_1t\langle x\rangle^{1+\frac km}}(e^{-tH_{k,m}}g)\big\Vert_{L^2(\mathbb R^n)}\le c_2\Vert g\Vert_{L^2(\mathbb R^n)}.
\end{equation}
To that end, we consider $c>0$ a positive constant whose value will be chosen later and the smooth function $\phi\in C^{\infty}(\mathbb R^n,\mathbb R)$ defined for all $x\in\mathbb R^n$ by
$$\phi(x) = c\langle x\rangle^{1+\frac km}.$$
In order to work with a smooth compactly support approximation of the function $\phi$, we consider the family of function $(\chi_{\varepsilon})_{\varepsilon>0}$ defined in \eqref{03022020E1} anew and we set $\phi_{\varepsilon} = \chi_{\varepsilon}\circ\phi$. In contrast to Subsection \ref{general}, where we used a Hilbert basis composed of eigenfunctions of the operator $H_{k,m}$, the strategy adopted here consists in directly manipulating the semigroup $(e^{-tH_{k,m}})_{t\geq0}$ through the following time-dependent functionals
\begin{equation}\label{16112020E2}
	F_{\varepsilon}(t) = \big\langle e^{-tH_{k,m}}g,e^{2t\phi_{\varepsilon}}e^{-tH_{k,m}}g\big\rangle_{L^2(\mathbb R^n)},\quad \varepsilon>0,\ t\geq0,\ g\in L^2(\mathbb R^n).
\end{equation}
However, in oder to justify that the functionals $F_{\varepsilon}$ are well-defined on $[0,+\infty)$, we need to call on Corollary \ref{02102020C1} proven in Subsection \ref{general}, which states in particular that 
$$\forall t>0, \forall g\in L^2(\mathbb R^n),\quad e^{-tH_{k,m}}g\in\mathscr S(\mathbb R^n).$$
Moreover, these functionals are differentiable on $(0,+\infty)$ and their derivatives are given for all $\varepsilon>0$, $t>0$ and $g\in L^2(\mathbb R^n)$ by
$$F_{\varepsilon}'(t) = -2\big\langle H_{k,m}e^{-tH_{k,m}}g,e^{2t\phi_{\varepsilon}}e^{-tH_{k,m}}g\big\rangle_{L^2(\mathbb R^n)} + 2\big\langle e^{-tH_{k,m}}g,\phi_{\varepsilon}e^{2t\phi_{\varepsilon}}e^{-tH_{k,m}}g\big\rangle_{L^2(\mathbb R^n)}.$$
By using the definition of the operator $H_{k,m}$, we can expand the above equality for all $\varepsilon>0$, $t>0$ and $g\in L^2(\mathbb R^n)$,
\begin{multline*}
	F'_{\varepsilon}(t) = -2\big\langle\vert x\vert^{2k}e^{-tH_{k,m}}g,e^{2t\phi_{\varepsilon}}e^{-tH_{k,m}}g\big\rangle_{L^2(\mathbb R^n)} 
- 2\big\langle(-\Delta)^m e^{-tH_{k,m}}g,e^{2t\phi_{\varepsilon}}e^{-tH_{k,m}}g\big\rangle_{L^2(\mathbb R^n)} \\[5pt]
+ 2\big\langle e^{-tH_{k,m}}g,\phi_{\varepsilon}e^{2t\phi_{\varepsilon}}e^{-tH_{k,m}}g\big\rangle_{L^2(\mathbb R^n)}.
\end{multline*}
We recall from Proposition \ref{19022020P1} (in the particular case where $\sigma=1$) that there exists a positive constant $c_0>0$ depending on the function $\phi$ such that for all $0<\varepsilon\le1$, $0\le t\le1$ and $v\in\mathscr S(\mathbb R^n)$,
$$\big\langle e^{t\phi_{\varepsilon}}(-\Delta)^m(e^{-t\phi_{\varepsilon}}v),v\big\rangle_{L^2(\mathbb R^n)}
+c_0\Big(\Vert v\Vert^2_{L^2(\mathbb R^n)} + t\big\Vert\langle x\rangle^kv\big\Vert^2_{L^2(\mathbb R^n)}\Big)\geq0.$$
By applying this estimate to the Schwartz functions $v= e^{2t\phi_{\varepsilon}}e^{-tH}g$, we deduce that for all $0<\varepsilon\le1$, $0<t\le1$ and $g\in L^2(\mathbb R^n)$,  
\begin{multline*}
	F'_{\varepsilon}(t) \le - 2\big\langle\vert x\vert^{2k}e^{-tH_{k,m}}g,e^{2t\phi_{\varepsilon}}e^{-tH_{k,m}}g\big\rangle_{L^2(\mathbb R^n)} 
	+ 2c_0t\big\Vert\langle x\rangle^ke^{t\phi_{\varepsilon}}e^{-tH_{k,m}}g\big\Vert^2_{L^2(\mathbb R^n)} \\[5pt]
	+ 2c_0\big\Vert e^{t\phi_{\varepsilon}}e^{-tH_{k,m}}g\big\Vert^2_{L^2(\mathbb R^n)} + 2\big\langle e^{-tH_{k,m}}g,\phi_{\varepsilon}e^{2t\phi_{\varepsilon}}e^{-tH_{k,m}}g\big\rangle_{L^2(\mathbb R^n)}.
\end{multline*}
This estimate also takes the following integral form
$$F'_{\varepsilon}(t)\le2\int_{\mathbb R^n}\big(c_0 + \phi_{\varepsilon}(x) + c_0t\langle x\rangle^{2k} - \vert x\vert^{2k}\big)\ e^{2t\phi_{\varepsilon}(x)}\big\vert(e^{-tH_{k,m}}g)(x)\big\vert^2\,\mathrm dx.$$
We deduce from the definition \eqref{03022020E1} of the functions $\chi_{\varepsilon}$ that there exists a positive constant $c_1>0$ such that for all $\varepsilon>0$ and $x\in\mathbb R^n$,
$$\chi_{\varepsilon}(x)=\frac1{\varepsilon}\chi(\varepsilon x)\le c_1x.$$
In particular, we get that for all $\varepsilon>0$ and $x\in\mathbb R^n$,
\begin{equation}\label{16112020E1}
	\phi_{\varepsilon}(x) = \chi_{\varepsilon}(\phi(x))\le c_1\phi(x).
\end{equation}
This implies that for all $0<\varepsilon\le1$, $0<t\le1$ and $g\in L^2(\mathbb R^n)$,
\begin{equation}\label{05122019E9}
	F'_{\varepsilon}(t)\le2\int_{\mathbb R^n}\big(c_0 + c_1\phi(x) + c_0t\langle x\rangle^{2k} - \vert x\vert^{2k}\big)\ e^{2t\phi_{\varepsilon}(x)}\big\vert(e^{-tH_{k,m}}g)(x)\big\vert^2\ \mathrm dx.
\end{equation}
We will now distinguish two regions in $\mathbb R^n$, namely in a neighborhood and far from the origin, in order to control the term
$$c_0 + c_1\phi(x) + c_0t\langle x\rangle^{2k} - \vert x\vert^{2k}.$$
Let $r_0>0$ be a radius whose value will be chosen later. On the one hand, since the above term is continuous with respect to both variables $t$ and $x$, we get that there exists a positive constant $M_{r_0}>0$, depending on $r_0$, such that for all $0\le t\le 1$ and $x\in\mathbb R^n$ satisfying $\vert x\vert\le r_0$,
\begin{equation}\label{05122019E6}
	c_0 + c_1\phi(x) + c_0t\langle x\rangle^{2k} - \vert x\vert^{2k}\le M_{r_0}.
\end{equation}
On the other hand, by choosing $c<1/c_1$, the value of the constant $r_0\gg1$ can be adjusted large enough so that there exists a positive constant $t_0>0$ such that for all $0\le t\le t_0$ and $x\in\mathbb R^n$ satisfying $\vert x\vert\geq r_0$,
\begin{equation}\label{05122019E7}
	c_0 + c_1\phi(x) + c_0t\langle x\rangle^{2k} - \vert x\vert^{2k}\le 0.
\end{equation}
Indeed, notice that with this choice, $c_1c-1<0$, and that the inequality $1+k/m\le2k$ implies that for all $0\le t\le 1$,
\begin{multline*}
	c_0 + c_1\phi(x) + c_0t\langle x\rangle^{2k} - \vert x\vert^{2k}\le c_0 + c_1c\langle x\rangle^{2k} + c_0t\langle x\rangle^{2k}-\vert x\vert^{2k} \\
	\underset{\vert x\vert\rightarrow+\infty}{\sim}(c_0t+c_1c-1)\vert x\vert^{2k}.
\end{multline*}
The value of the radius $r_0\gg1$ is now fixed. We deduce from \eqref{05122019E9}, \eqref{05122019E6} and \eqref{05122019E7} that for all $\varepsilon>0$, $0<t\le t_0$ and $g\in L^2(\mathbb R^n)$, 
$$F'_{\varepsilon}(t)\le2M_{r_0}\int_{\vert x\vert\le r_0}e^{2t\phi_{\varepsilon}(x)}\big\vert(e^{-tH_{k,m}}g)(x)\big\vert^2\,\mathrm dx.$$
Moreover, by using \eqref{16112020E1} anew and the continuity of the function $\phi$, we can find a positive constant $M>0$ such that for all $0<\varepsilon\le1$, $0<t\le t_0$ and $g\in L^2(\mathbb R^n)$, 
\begin{equation}\label{13122019E2}
	F'_{\varepsilon}(t)\le M\int_{\vert x\vert\le r_0}\big\vert(e^{-tH_{k,m}}g)(x)\big\vert^2\,\mathrm dx\le M\Vert g\Vert^2_{L^2(\mathbb R^n)}.
\end{equation}
By using the definition \eqref{16112020E2} of the functionals $F_{\varepsilon}$ and integrating the above estimate, we deduce that for all $0<\varepsilon<\varepsilon_0$, $0\le t\le t_0$ and $g\in L^2(\mathbb R^n)$,
$$F_{\varepsilon}(t) = \big\Vert e^{t\phi_{\varepsilon}}e^{-tH_{k,m}}g\big\Vert^2_{L^2(\mathbb R^n)}\le (1+Mt)\Vert g\Vert^2_{L^2(\mathbb R^n)}\le(1+Mt_0)\Vert g\Vert^2_{L^2(\mathbb R^n)}.$$
Using Fatou's lemma as in the end of Section \ref{agmon} therefore ends the proof of the estimate \eqref{05122019E3} and the one of Theorem \ref{14102020E1}.

\section{A G{\aa}rding type inequality}\label{secgard}

Let $\phi$ be the smooth function defined in \eqref{05032020E1} and $(\phi_{\varepsilon})_{\varepsilon>0}$ be the family of $C^{\infty}_0(\mathbb R^n)$ functions given by \eqref{03022020E1}. This section is devoted to the proof of Proposition \ref{19022020P1}, which states that there exists a positive constant $c_0>0$ depending on the function $\phi$ such that for all $0<\varepsilon\le1$, $0\le t\le1$ and $v\in\mathscr S(\mathbb R^n)$,
\begin{equation}\label{06022020E1}
	\big\langle e^{t\phi_{\varepsilon}}(-\Delta)^m(e^{-t\phi_{\varepsilon}}v),v\big\rangle_{L^2(\mathbb R^n)}
	+c_0\Big(\Vert v\Vert^2_{L^2(\mathbb R^n)} + t\big\Vert\langle x\rangle^{\sigma k}v\big\Vert^2_{L^2(\mathbb R^n)}\Big)\geq0.
\end{equation}
We recall that for all $\varepsilon>0$ and $x\in\mathbb R^n$.
$$\phi(x) = \langle x\rangle^{\sigma(1+\frac km)},\quad\phi_{\varepsilon}(x) = (\chi_{\varepsilon}\circ\phi)(x)\quad\text{with}\quad\chi_{\varepsilon}(x) = \frac1{\varepsilon}\chi(\varepsilon x),$$
where $0\le\sigma\le1$ is a non-negative real number, $k,m\geq1$ are positive integers and $\chi\in C^{\infty}_0(\mathbb R,\mathbb R)$ is a cut-off odd function satisfying that $\chi(x) = x$ for all $0\le x\le1$, $\chi(x) = 0$ when $x\geq2$ and $\chi(x)\geq0$ for all $x\geq0$. A natural approach to prove the estimate \eqref{06022020E1} would be to compute explicitly the scalar product $\langle e^{t\phi_{\varepsilon}}(-\Delta)^m(e^{-t\phi_{\varepsilon}}v),v\rangle_{L^2(\mathbb R^n)}$ by using Leibniz' and Fa\'a di Bruno's formulas, and to manage all the terms appearing. However, due to the form of the general Fa\'a di Bruno's formula, these terms would not have a manageable form, and their study would be difficult to tackle, especially since we have to take into account the parameters $0\le t\le 1$ and $0<\varepsilon\le1$. This is the reason why we will use technics from symbolic calculus, which is particularly a well-adapted framework to prove G{\aa}rding type inequalities like \eqref{06022020E1}.

The first step consists in providing a more manageable form for the operators $e^{t\phi_{\varepsilon}}(-\Delta)^me^{-t\phi_{\varepsilon}}$ involved in the above estimate. To that end, we need to introduce the following commutator notation for all compactly supported smooth function $\chi\in C^{\infty}_0(\mathbb R^n)$ and all differential operator $P$,
$$[\chi,P] = \chi P - P\chi.$$
We also define $\ad^0_{\chi}(-\Delta)^m = (-\Delta)^m$ and for all $j\geq0$,
$$\ad^{j+1}_{\chi}(-\Delta)^m = [\chi,\ad^j_{\chi}(-\Delta)^m].$$

\begin{lem}\label{18082020L1} For all compactly supported smooth function $\chi\in C^{\infty}_0(\mathbb R^n)$, the following formula between differential operators holds,
\begin{equation}\label{18082020E2}
	\forall t\in\mathbb R,\quad e^{t\chi}(-\Delta)^me^{-t\chi} = \sum_{j=0}^{2m}\frac{t^j}{j!}\ad^j_{\chi}(-\Delta)^m.
\end{equation}
\end{lem}

\begin{proof} Let $\chi\in C^{\infty}_0(\mathbb R^n)$ be a fixed compactly supported smooth function. First notice from a straightforward induction that
\begin{equation}\label{18082020E1}
	\forall j\geq0,\forall t\in\mathbb R,\quad\partial^j_t(e^{t\chi}(-\Delta)^me^{-t\chi}) = e^{t\chi}\ad^j_{\chi}(-\Delta)^me^{-t\chi}.
\end{equation}
Moreover, the differential operator $\ad^j_{\chi}(-\Delta)^m$ is of order $\max(2m-j,0)$, which implies that the derivatives \eqref{18082020E1} are equal to zero provided $j\geq2m+1$. We therefore deduce from Taylor's formula applied to the analytic functions $t\in\mathbb R\mapsto e^{t\chi}(-\Delta)^m(e^{-t\chi}v)(x)$, with $v\in\mathscr S(\mathbb R^n)$ and $x\in\mathbb R^n$, that the equality \eqref{18082020E2} actually holds.
\end{proof}

We deduce from Lemma \ref{18082020L1} that for all $\varepsilon>0$, $t\geq0$ and $v\in\mathscr S(\mathbb R^n)$,
\begin{equation}\label{18022020E1}
	\big\langle e^{t\phi_{\varepsilon}}(-\Delta)^m(e^{-t\phi_{\varepsilon}}v),v\big\rangle_{L^2(\mathbb R^n)} = \Vert v\Vert^2_{\dot H^m(\mathbb R^n)} + \sum_{j=1}^{2m}\frac{t^j}{j!}\big\langle\ad^j_{\phi_{\varepsilon}}(-\Delta)^mv,v\big\rangle_{L^2(\mathbb R^n)}.
\end{equation}
The objective is now to control each term appearing in the above sum. Precisely, we will prove that there exists a positive constant $c>0$ such that for all $\eta>0$ there exists another positive constant $C_{\eta}>0$ such that for all $1\le j\le 2m$, $0<\varepsilon\le1$ and $v\in\mathscr S(\mathbb R^n)$,
\begin{equation}\label{18022020E2}
	\big\vert\big\langle\ad^j_{\phi_{\varepsilon}}(-\Delta)^mv,v\big\rangle_{L^2(\mathbb R^n)}\big\vert\le c\bigg(\frac1{\eta}\big\Vert\langle x\rangle^{\sigma k}v\big\Vert^2_{L^2(\mathbb R^n)} + \eta\Vert v\Vert^2_{\dot H^m(\mathbb R^n)} + C_{\eta}\Vert v\Vert^2_{L^2(\mathbb R^n)}\bigg).
\end{equation}
Adjusting the value of $0<\eta\ll1$, we derive the estimate \eqref{06022020E1} from \eqref{18022020E1} and \eqref{18022020E2}. 

As announced in the beginning of this section, the strategy to obtain \eqref{18022020E2} is to use results from the theory of symbolic calculus, of which we now recall some basic notions and notations. Given $\Phi,\Psi\in C^0(\mathbb R^{2n})$ some sub-linear and temperate weights, and $M\in C^0(\mathbb R^{2n})$ another temperate weight, we define the symbol class $S(M ; \Phi,\Psi)$ as the set of all smooth functions $a\in C^{\infty}(\mathbb R^{2n})$ satisfying that for all $(\alpha,\beta)\in\mathbb N^{2n}$, there exists a positive constant $c_{\alpha,\beta}>0$ such that for all $(x,\xi)\in\mathbb R^{2n}$,
\begin{equation}\label{16092020E1}
	\vert(\partial^{\alpha}_x\partial^{\beta}_{\xi}a)(x,\xi)\vert\le c_{\alpha,\beta}M(x,\xi)\Psi(x,\xi)^{-\vert\alpha\vert}\Phi(x,\xi)^{-\vert\beta\vert}.
\end{equation}
We refer to \cite{MR2668420} (page 19) where the notions of sub-linear and temperate weights are defined. We also recall from the very same reference that examples of temperate weights are given by $\langle x\rangle^m$ or $\langle\xi\rangle^m$, seen as functions of $\mathbb R^{2n}$, with $m\in\mathbb R$. Associated to any $a\in S(M ; \Phi,\Psi)$ is the pseudodifferential operator $a^w$ defined by the Weyl quantization of the symbol $a$, that is, formally,
$$(a^wu)(x) = \frac1{(2\pi)^n}\iint_{\mathbb R^{2n}}e^{i(x-y)\cdot\xi}a\bigg(\frac{x+y}2,\xi\bigg)u(y)\,\mathrm dy\mathrm d\xi.$$
We refer to \cite{MR2668420} (Formula 1.2.3) for a rigorous definition of the operator $a^w$. According to \cite{MR2668420} (Theorem 1.2.17), for all symbols $a\in S(M_1 ; \Phi,\Psi)$ and $b\in S(M_2 ; \Phi,\Psi)$, the composition $a^wb^w = (a\ \sharp\ b)^w$ is also a pseudodifferential operator, the associated symbol $a\ \sharp\ b\in S(M_1M_2 ; \Phi,\Psi)$ being given for all $(x,\xi)\in\mathbb R^{2n}$ by
\begin{equation}\label{16112020E4}
	(a\ \sharp\ b)(x,\xi) = e^{\frac i2(D_y\cdot D_{\eta} - D_x\cdot D_{\xi})}a(x,\eta)b(y,\xi)\Big\vert_{(y,\eta) = (x,\xi)}.
\end{equation}
Moreover, we have the following asymptotic expansion
\begin{equation}\label{16112020E3}
	(a\ \sharp\ b)(x,\xi) \sim \sum_{\alpha,\beta}\frac{(-1)^{\vert\beta\vert}}{2^{\vert\alpha+\beta\vert}\alpha!\beta!}(\partial^{\alpha}_{\xi}D^{\beta}_xa)(x,\xi)(\partial^{\beta}_{\xi}D^{\alpha}_xb)(x,\xi),
\end{equation}
which is an equality when $a$ or $b$ is a polynomial as mentioned in \cite{MR2304165} (Theorem 18.5.4). This is also an exact formula when the symbol $a$ only depends on the space variable $x\in\mathbb R^n$ and the symbol $b$ is a polynomial with respect to the frequency variable $\xi\in\mathbb R^n$, see e.g. \cite{MR2599384} (Formula (2.1.28)) for an expression of the remainder. This asymptotic expansion will be widely used in the following (in fact, this will be an exact formula in the futur applications in this section).

For all $0\le j\le 2m$ and $\varepsilon>0$, we consider $\sigma_{j,\varepsilon}$ the Weyl symbol of the differential operator $\ad^j_{\phi_{\varepsilon}}(-\Delta)^m$. We now need to determine in which class the symbols $\sigma_{j,\varepsilon}$ belong. This is done thanks to the following two lemmas.

\begin{lem}\label{10092020L1} The following relation recurrence holds for all $\varepsilon>0$ and $0\le j\le 2m-1$,
\begin{equation}\label{10092020E1}
	\sigma_{j+1,\varepsilon} = -\sum_l\frac1{2^{l-1}}\sum_{\vert\alpha\vert=l}\frac1{\alpha!}(D^{\alpha}_x\phi_{\varepsilon})(\partial^{\alpha}_{\xi}\sigma_{j,\varepsilon}),
\end{equation}
the sum being taken over all the odd integers $l$ satisfying $1\le l\le 2m-j$.
\end{lem}

\begin{proof} Let $\varepsilon>0$ fixed all along this proof. We establish the relation \eqref{10092020E1} by induction, also checking that each symbol $\sigma_{j,\varepsilon}$ is a polynomial of degree $2m-j$ with respect to the $\xi$ variable. Let us begin with the case $j=0$. Since $\vert\xi\vert^{2m}$ is a polynomial, the composition formula \eqref{16112020E3} shows that the Weyl symbol of the operator $\phi_{\varepsilon}(-\Delta)^m$ is exactly given by
\begin{align*}
	\phi_{\varepsilon}\ \sharp\ \vert\xi\vert^{2m}& = \sum_{l=0}^{2m}\frac1{2^l}\sum_{\vert\alpha\vert + \vert\beta\vert = l}\frac{(-1)^{\vert\beta\vert}}{\alpha!\beta!}(\partial^{\alpha}_{\xi}D^{\beta}_x\phi_{\varepsilon})(\partial^{\beta}_{\xi}D^{\alpha}_x\vert\xi\vert^{2m}) \\
	& = \sum_{l=0}^{2m}\frac1{2^l}\sum_{\vert\beta\vert = l}\frac{(-1)^l}{\beta!}(D^{\beta}_x\phi_{\varepsilon})(\partial^{\beta}_{\xi}\vert\xi\vert^{2m}).
\end{align*}
Similarly, the Weyl symbol of the operator $(-\Delta)^m\phi_{\varepsilon}$ is given by
$$\vert\xi\vert^{2m}\ \sharp\ \phi_{\varepsilon} = \sum_{l=0}^{2m}\frac1{2^l}\sum_{\vert\alpha\vert = l}\frac1{\alpha!}(\partial^{\alpha}_{\xi}\vert\xi\vert^{2m})(D^{\alpha}_x\phi_{\varepsilon}).$$
We therefore deduce the following expression for the symbol $\sigma_{1,\varepsilon}$,
\begin{align*}
	\sigma_{1,\varepsilon} = \phi_{\varepsilon}\ \sharp\ \vert\xi\vert^{2m} -  \vert\xi\vert^{2m}\ \sharp\ \phi_{\varepsilon}
	& = \sum_{l=0}^{2m}\frac1{2^l}\sum_{\vert\alpha\vert = l}\frac{(-1)^l-1}{\alpha!}(D^{\alpha}_x\phi_{\varepsilon})(\partial^{\alpha}_{\xi}\vert\xi\vert^{2m}) \\[5pt]
	& = -\sum_l\frac1{2^{l-1}}\sum_{\vert\alpha\vert = l}\frac1{\alpha!}(D^{\alpha}_x\phi_{\varepsilon})(\partial^{\alpha}_{\xi}\vert\xi\vert^{2m}),
\end{align*}
the last sum being taken over all the odd integers $l$ satisfying $1\le l\le 2m$. The above formula shows that the symbol $\sigma_{1,\varepsilon}$ is a polynomial of degree $2m-1$ with respect to the $\xi$ variable. Since $\ad^0_{\phi_{\varepsilon}}(-\Delta)^m$ is the operator $(-\Delta)^m$ by definition, the basic case of the induction is ended. We now consider $j\geq1$ and assume that formula \eqref{10092020E1} holds for $j-1$. Since the function $\phi_{\varepsilon}$ only depends on the space variable $x\in\mathbb R^n$ and that the symbol $\sigma_{j,\varepsilon}$ is a polynomial of degree $2m-j$ with respect to the frequency variable $\xi\in\mathbb R^n$ according to the induction hypothesis, we deduce the composition formula \eqref{16112020E3} anew that the symbol of the operator $\phi_{\varepsilon}\ad^j_{\phi_{\varepsilon}}(-\Delta)^m$ is given by 
$$\phi_{\varepsilon}\ \sharp\ \sigma_{j,\varepsilon} = \sum_{l=0}^{2m-j}\frac1{2^l}\sum_{\vert\beta\vert = l}\frac{(-1)^l}{\beta!}(D^{\beta}_x\phi_{\varepsilon})(\partial^{\beta}_{\xi}\sigma_{j,\varepsilon}).$$
Similarly as in the basic case, we then obtain the following formula for the symbol $\sigma_{j+1,\varepsilon}$
$$\sigma_{j+1,\varepsilon} = -\sum_l\frac1{2^{l-1}}\sum_{\vert\alpha\vert=l}\frac1{\alpha!}(D^{\alpha}_x\phi_{\varepsilon})(\partial^{\alpha}_{\xi}\sigma_{j,\varepsilon}),$$
the sum being taken over all the odd integers $l$ satisfying $1\le l\le 2m-j$, which is the relation we aimed at obtaining. Notice that it implies that the symbol $\sigma_{j+1,\varepsilon}$ is a polynomial of degree $2m-j-1$ with respect to the frequency variable $\xi\in\mathbb R^n$. This ends the proof of Lemma \ref{10092020L1}.
\end{proof}

\begin{lem}\label{11092020L1} For all $(\alpha,\beta)\in\mathbb N^{2n}$, there exists a positive constant $c_{\alpha,\beta}>0$ such that for all $1\le j\le 2m$, $0<\varepsilon\le1$ and $(x,\xi)\in\mathbb R^{2n}$,
\begin{equation}\label{11092020E4}
	\big\vert(D^{\alpha}_x\partial^{\beta}_{\xi}\sigma_{j,\varepsilon})(x,\xi)\big\vert\le c_{\alpha,\beta}\langle x\rangle^{\frac{\sigma kj}m}\langle\xi\rangle^{2m-j-\vert\beta\vert}.
\end{equation}
\end{lem}

\begin{proof} We proceed by induction, beginning with the case $j=1$. We deduce from Lemma \ref{10092020L1} that for all $\varepsilon>0$, the derivatives of the symbol $\sigma_{1,\varepsilon}$ are given by
\begin{equation}\label{11092020E3}
	D^{\alpha}_x\partial^{\beta}_{\xi}\sigma_{1,\varepsilon} = -\sum_l\frac1{2^{l-1}}\sum_{\vert\gamma\vert = l}\frac1{\alpha!}(D^{\alpha+\gamma}_x\phi_{\varepsilon})(\partial^{\beta+\gamma}_{\xi}\vert\xi\vert^{2m}),\quad (\alpha,\beta)\in\mathbb N^{2n},
\end{equation}
the sum being taken over all the odd integers $l$ satisfying $1\le l\le 2m$. On the one hand, we need to bound the derivatives of the functions $\phi_{\varepsilon}$. By definition of the function $\chi$ and the functions $\chi_{\varepsilon}$ in \eqref{03022020E1}, we get that for all $0<\varepsilon\le1$, we have $\Vert\chi_{\varepsilon}\Vert_{L^{\infty}(\mathbb R^n)}\le1$ and
\begin{equation}\label{11092020E1}
	\forall p\geq1,\quad \big\Vert\chi^{(p)}_{\varepsilon}\big\Vert_{L^{\infty}(\mathbb R^n)}\le\varepsilon^{p-1}\big\Vert\chi^{(p)}\big\Vert_{L^{\infty}(\mathbb R^n)}\le\big\Vert\chi^{(p)}\big\Vert_{L^{\infty}(\mathbb R^n)}.
\end{equation}
Notice in particular that the derivatives of the function $\chi_{\varepsilon}$ are uniformly bounded with respect to the parameter $0<\varepsilon\le1$. Moreover, the function $\phi$ is defined in \eqref{05032020E1} as a Japanese bracket, which implies that for all $\rho\in\mathbb N^n$, there exists a positive constant $c_{\rho}>0$ such that for all $x\in\mathbb R^n$,
\begin{equation}\label{11092020E2}
	\big\vert(D^{\rho}_x\phi)(x)\big\vert\le c_{\rho}\langle x\rangle^{\sigma(1+\frac km)-\vert\rho\vert}.
\end{equation}
Since $\phi_{\varepsilon} = \chi_{\varepsilon}\circ\phi$ and $0\le\sigma\le1$, we deduce from \eqref{11092020E1}, \eqref{11092020E2} and the chain rule that for all $\alpha,\gamma\in\mathbb N^n$ with $\vert\gamma\vert\geq1$, there exists a positive constant $c_{\alpha,\gamma}>0$ such that for all $0<\varepsilon\le 1$ and $x\in\mathbb R^n$,
\begin{equation}\label{11092020E5}
	\big\vert(D^{\alpha+\gamma}_x\phi_{\varepsilon})(x)\big\vert\le c_{\alpha,\gamma}\langle x\rangle^{\sigma(1+\frac km)-1}\le c_{\alpha,\gamma}\langle x\rangle^{\frac{\sigma k}m}.
\end{equation}
On the other hand, we get that for all $\beta,\gamma\in\mathbb N^n$ with $\vert\gamma\vert\geq1$, there exists a positive constant $c_{\beta,\gamma}>0$ such that for all $\xi\in\mathbb R^n$,
$$\big\vert\partial^{\beta+\gamma}_{\xi}\vert\xi\vert^{2m}\big\vert\le c_{\beta,\gamma}\langle\xi\rangle^{2m-\vert\beta\vert-\vert\gamma\vert}\le c_{\beta,\gamma}\langle\xi\rangle^{2m-1-\vert\beta\vert}.$$
In view of \eqref{11092020E3} (notice that the first sum starts at $l=1$) and the two above estimates, the induction is ended in the basic case. We now consider $2\le j\le 2m$ assume that formula \eqref{11092020E4} holds for $j-1$. We deduce from Lemma \ref{10092020L1} anew that for all $\varepsilon>0$, the derivatives of the symbol $\sigma_{j+1,\varepsilon}$ are given by
\begin{equation}\label{11092020E7}
	D^{\alpha}_x\partial^{\beta}_{\xi}\sigma_{j+1,\varepsilon} = -\sum_l\frac1{2^{l-1}}\sum_{\vert\gamma\vert = l}\frac1{\alpha!}D^{\alpha}_x((D^{\gamma}_x\phi_{\varepsilon})(\partial^{\beta+\gamma}_{\xi}\sigma_{j,\varepsilon})),\quad (\alpha,\beta)\in\mathbb N^{2n},
\end{equation}
the sum being taken over all the odd integers $l$ satisfying $1\le l\le 2m-j$. Moreover, Leibniz' formula implies that 
\begin{equation}\label{11092020E6}
	D^{\alpha}_x((D^{\gamma}_x\phi_{\varepsilon})(\partial^{\beta+\gamma}_{\xi}\sigma_{j,\varepsilon})) = \sum_{\delta\le\alpha}\binom{\alpha}{\delta}(D^{\delta+\gamma}_x\phi_{\varepsilon})(D^{\alpha-\delta}_x\partial^{\beta+\gamma}_{\xi}\sigma_{j,\varepsilon}).
\end{equation}
We deduce from \eqref{11092020E5} and the induction hypothesis that for all $(\alpha,\beta,\gamma,\delta)\in\mathbb N^{4n}$, with $\vert\gamma\vert\geq1$ and $\delta\le\gamma$, there exists a positive constant $c_{\alpha,\beta,\gamma,\delta}>0$ such that for all $0<\varepsilon\le 1$ and $(x,\xi)\in\mathbb R^{2n}$,
\begin{align*}
	\big\vert(D^{\delta+\gamma}_x\phi_{\varepsilon})(x,\xi)(D^{\alpha-\delta}_x\partial^{\beta+\gamma}_{\xi}\sigma_{j,\varepsilon})(x,\xi)\big\vert
	& \le c_{\alpha,\beta,\gamma,\delta}\langle x\rangle^{\frac{\sigma k}m}\langle x\rangle^{\frac{\sigma kj}m}\langle\xi\rangle^{2m-j-\vert\beta\vert-\vert\gamma\vert} \\
	& \le c_{\alpha,\beta,\gamma,\delta}\langle x\rangle^{\frac{\sigma k(j+1)}m}\langle\xi\rangle^{2m-(j+1)-\vert\beta\vert},
\end{align*}
since $\vert\gamma\vert\geq1$. In view of this estimate, formula \eqref{11092020E7} (notice that the first sum starts at $l=1$ anew) and \eqref{11092020E6}, the induction is now ended.
\end{proof}

By using the notations for symbol classes introduced above, Lemma \ref{11092020L1} shows that for all $1\le j\le 2m$ and $0<\varepsilon\le1$, the symbol $\sigma_{j,\varepsilon}$ belongs to the following class
\begin{equation}\label{17092020E4}
	\sigma_{j,\varepsilon}\in S(\langle x\rangle^{\frac{\sigma kj}m}\langle\xi\rangle^{2m-j} ; \langle\xi\rangle,1),
\end{equation}
with uniform estimates of the associated seminorms \eqref{16092020E1} with respect to the parameter $0<\varepsilon\le1$. This property is the key stone of the proof of the next lemma, which provides a first bound of the quantities we aim at controlling.

\begin{lem}\label{11092020L2} There exists a positive constant $c>0$ such that for all $1\le j\le 2m$, $0<\varepsilon\le 1$ and $v\in\mathscr S(\mathbb R^n)$,
$$\big\vert\big\langle\ad^j_{\phi_{\varepsilon}}(-\Delta)^mv,v\big\rangle_{L^2(\mathbb R^n)}\big\vert\le c\big\Vert\langle x\rangle^{\frac{\sigma kj}{2m}}\langle D_x\rangle^{m-\frac j2}v\big\Vert^2_{L^2(\mathbb R^n)}.$$
\end{lem}

\begin{proof} Let $1\le j\le 2m$ fixed all along the proof. We deduce from the property \eqref{17092020E4} and the composition formula \eqref{16112020E3} that
$$\langle x\rangle^{-\frac{\sigma kj}{2m}}\ \sharp\ \langle\xi\rangle^{-m+\frac j2}\ \sharp\ \sigma_{j,\varepsilon}\ \sharp\ \langle\xi\rangle^{-m+\frac j2}\ \sharp\ \langle x\rangle^{-\frac{\sigma kj}{2m}}\in S(1 ; \langle\xi\rangle, 1)\subset C^{\infty}_b(\mathbb R^{2n}),$$
with uniform estimates of the associated seminorms \eqref{16092020E1} with respect to the parameter $0<\varepsilon\le1$. It therefore follows from a quantitative version of the Calder\'on-Vaillancourt's theorem, see e.g. \cite{MR1696697} (Theorem 1.2), that the following operator is bounded on $L^2(\mathbb R^n)$,
$$\langle x\rangle^{-\frac{\sigma kj}{2m}}\langle D_x\rangle^{-m+\frac j2}\sigma_{j,\varepsilon}^w\langle D_x\rangle^{-m+\frac j2}\langle x\rangle^{-\frac{\sigma kj}{2m}}: L^2(\mathbb R^n)\rightarrow L^2(\mathbb R^n),$$
and its norm operator can bounded uniformly with respect to $0<\varepsilon\le 1$. As a consequence, there exists a positive constant $c>0$ such that for all $0<\varepsilon\le1$ and $v\in\mathscr S(\mathbb R^n)$,
$$\big\vert\big\langle\sigma_{j,\varepsilon}^w\langle D_x\rangle^{-m+\frac j2}\langle x\rangle^{-\frac{\sigma kj}{2m}}v,\langle D_x\rangle^{-m+\frac j2}\langle x\rangle^{-\frac{\sigma kj}{2m}}v\big\rangle_{L^2(\mathbb R^n)}\big\vert\le c\Vert v\Vert^2_{L^2(\mathbb R^n)}.$$
A straightforward change of variable therefore ends the proof of Lemma \ref{11092020L2}, since $\sigma_{j,\varepsilon}$ is the Weyl symbol of the operator $\ad^j_{\phi_{\varepsilon}}(-\Delta)^m$ by definition.
\end{proof}

Lemma \ref{11092020L2} implies that now, we only need to prove that there exists a positive constant $c>0$ such that for all $\eta>0$ there exists another positive constant $C_{\eta}>0$ such that for all $1\le j\le 2m$ and $v\in\mathscr S(\mathbb R^n)$,
$$\big\Vert\langle x\rangle^{\frac{\sigma kj}{2m}}\langle D_x\rangle^{m-\frac j2}v\big\Vert^2_{L^2(\mathbb R^n)}\le c\bigg(\frac1{\eta}\big\Vert\langle x\rangle^{\sigma k}v\big\Vert^2_{L^2(\mathbb R^n)} + \eta\Vert v\Vert^2_{\dot H^m(\mathbb R^n)} + C_{\eta}\Vert v\Vert^2_{L^2(\mathbb R^n)}\bigg),$$
to derive the estimate \eqref{18022020E2} (we got rid of the parameter $0<\varepsilon\le1$). Notice that the composition formula \eqref{16112020E4} allows to consider the symbol 
\begin{equation}\label{17092020E3}
	a_0\in S(\langle x\rangle^{\frac{\sigma kj}m}\langle\xi\rangle^{2m-j} ; \langle\xi\rangle, \langle x\rangle),
\end{equation}
satisfying
$$a_0^w = \langle D_x\rangle^{m-\frac j2}\langle x\rangle^{\frac{\sigma kj}m}\langle D_x\rangle^{m-\frac j2}.$$
The estimate we aim at proving is therefore the following
\begin{equation}\label{17092020E1}
	\big\vert\big\langle a_0^wv,v\big\rangle_{L^2(\mathbb R^n)}\big\vert\le c\bigg(\frac1{\eta}\big\Vert\langle x\rangle^{\sigma k}v\big\Vert^2_{L^2(\mathbb R^n)} + \eta\Vert v\Vert^2_{\dot H^m(\mathbb R^n)} + C_{\eta}\Vert v\Vert^2_{L^2(\mathbb R^n)}\bigg).
\end{equation}
Moreover, the symbol $\langle x\rangle^{\frac{\sigma kj}m}\langle\xi\rangle^{2m-j}$ can be written in the following way
\begin{equation}\label{17092020E5}
	\langle x\rangle^{\frac{\sigma kj}m}\langle\xi\rangle^{2m-j} = \langle x\rangle^{\frac{2\sigma k}{p_j}}\langle\xi\rangle^{\frac{2m}{q_j}}\quad\text{with}\quad p_j = \frac{2m}j,\ q_j = \frac{2m}{2m-j}.
\end{equation}
Notice that $p_j$ and $q_j$ are H\"older conjugates, that is, $1/p_j+1/q_j = 1$. This observation motivates the introduction of the following symbol classes for all positive real numbers $g,h>0$ and $p,q\geq1$ (not necessary H\"older conjugates in general),
$$S^{g,h}_{p,q} = S(\langle x\rangle^{\frac gp}\langle\xi\rangle^{\frac hq} ; \langle\xi\rangle,\langle x\rangle).$$
Young's inequality implies that when $1/p+1/q=1$, any symbol $a\in S^{2g,2h}_{p,q}$ satisfies
$$\forall\eta>0,\forall(x,\xi)\in\mathbb R^{2n},\quad\vert a(x,\xi)\vert\lesssim\frac1{\eta}\langle x\rangle^{2g} + \eta\langle\xi\rangle^{2h},$$
the constant only depending on the real numbers $g,h,p,q$ (and not on $\eta>0$). Applying a formal G{\aa}rding type inequality, we would conjecture that an estimate of the following form could hold for all $\eta>0$ and $v\in\mathscr S(\mathbb R^n)$,
\begin{equation}\label{17092020E2}
	\big\vert\big\langle a^wv,v\big\rangle_{L^2(\mathbb R^n)}\big\vert\lesssim\frac1{\eta}\big\Vert\langle x\rangle^gv\big\Vert^2_{L^2(\mathbb R^n)} + \eta\Vert v\Vert^2_{H^h(\mathbb R^n)} + C_{\eta}\Vert v\Vert^2_{L^2(\mathbb R^n)}.
\end{equation}
This is exactly the type of estimate we aim at proving. In order to make this formal derivation rigorous, we introduce the notion of anti-Wick quantization, following \cite{MR2668420} (Section 1.7), which has the advantage to preserve positivity, in contrast to the Weyl quantization. Given a tempered symbol $a\in\mathscr S'(\mathbb R^n)$, we define the anti-Wick operator $A_a$ with symbol $a$ as the map $\mathscr S(\mathbb R^n)\rightarrow\mathscr S'(\mathbb R^n)$ given by 
$$A_au = (2\pi)^{-n}V^*(aVu),\quad u\in\mathscr S(\mathbb R^n),$$
where $V$ denotes the short-time Fourier transform, see e.g. \cite{MR2668420} (Definition 1.7.1). The idea to use anti-Wick operators in this context was suggested to the author by J. Bernier who (with co-authors) used this notion in the note \cite{BCC} to obtain microlocal estimates. In the next two lemmas, we explicit the relationship that exists between the anti-Wick operators with symbols in the class $S^{g,h}_{p,q}$ and the Weyl quantization of those symbols. Their proofs are inspired by the one of \cite{BCC} (Lemma 1).

\begin{lem}\label{09092020L1} Let $g,h>0$ and $p,q\geq1$ be positive real numbers. For all symbol $a\in S^{g,h}_{p,q}$, there exists remainders $r_1\in S^{g-p,h}_{p,q}$ and $r_2\in S^{g,h-q}_{p,q}$ such that
$$A_a = a^w + r^w_1 + r^w_2,$$
where $A_a$ denotes the anti-Wick operator with symbol $a$.
\end{lem}

\begin{proof} By applying \cite{MR2668420} (Proposition 1.7.9), we know that the Weyl symbol of the operator $A_a$ is the symbol $b\in C^{\infty}(\mathbb R^{2n})$ given by
\begin{equation}\label{defb}
	b(x,\xi) = \frac1{\pi^n}\iint_{\mathbb R^{2n}} a(\tilde x,\tilde\xi)e^{-\vert x-\tilde x\vert^2-\vert\xi-\tilde\xi\vert^2}\ \mathrm d\tilde x\mathrm d\tilde\xi,\quad (x,y)\in\mathbb R^{2n}.
\end{equation}
Applying Taylor's formula with remainder term to the symbol $a$ in $(x,\xi)\in\mathbb R^{2n}$ leads to
\begin{multline*}
	a(\tilde x,\tilde\xi) = a(x,\xi) + \int_0^1\nabla_xa(x+t(\tilde x-x),\xi+t(\tilde\xi-\xi))\cdot(\tilde x-x)\ \mathrm dt \\
	+ \int_0^1\nabla_{\xi}a(x+t(\tilde x-x),\xi+t(\tilde\xi-\xi))\cdot(\tilde\xi-\xi)\ \mathrm dt.
\end{multline*}
Plugging this expansion in the definition \eqref{defb} of the symbol $b$ and making the change of variables $(\tilde x,\tilde\xi)\leftarrow (\tilde x-x,\tilde\xi -\xi)$ motivates to introduce the two following remainders
$$r_1(x,\xi) = \frac1{\pi^n}\iint_{\mathbb R^{2n}}\int_0^1\nabla_xa(x+t\tilde x,\xi+t\tilde\xi)\cdot\tilde x\ e^{-\vert\tilde x\vert^2-\vert\tilde\xi\vert^2}\ \mathrm dt\mathrm d\tilde x\mathrm d\tilde\xi,$$
and
$$r_2(x,\xi) = \frac1{\pi^n}\iint_{\mathbb R^{2n}}\int_0^1\nabla_{\xi}a(x+t\tilde x,\xi+t\tilde\xi)\cdot\tilde\xi\ e^{-\vert\tilde x\vert^2-\vert\tilde\xi\vert^2}\ \mathrm dt\mathrm d\tilde x\mathrm d\tilde\xi.$$
Indeed, with these definitions of $r_1$ and $r_2$, we have
$$A_a = b^w = a^w + r^w_1 + r^w_2.$$
Let us check that $r_1\in S^{g-p,h}_{p,q}$. By symmetry, we will also have that $r_2\in S^{g,h-q}_{p,q}$. We just need to prove that for all $(\alpha,\beta)\in\mathbb N^{2n}$, there exists a positive constant $c_{\alpha,\beta}>0$ such that for all $(x,\xi)\in\mathbb R^{2n}$,
\begin{equation}\label{09092020E1}
	\big\vert(\partial_x^{\alpha}\partial_\xi^{\beta}r_1)(x,\xi)\big\vert\le c_{\alpha,\beta}\langle x\rangle^{\frac{g-p}p-\vert\alpha\vert}\langle\xi\rangle^{\frac hq-\vert\beta\vert}.
\end{equation}
By definition of the symbol class $S^{g,h}_{p,q}$, we get that for all $(\delta,\gamma)\in\mathbb N^{2n}$, there exists a positive constant $c_{\delta,\gamma}>0$ such that for all $(x,\xi)\in\mathbb R^{2n}$,
$$\big\vert(\partial_x^{\delta}\partial_\xi^{\gamma}a)(x,\xi)\big\vert\le c_{\delta,\gamma}\langle x\rangle^{\frac gp-\vert\delta\vert}\langle\xi\rangle^{\frac hq-\vert\gamma\vert}.$$
We therefore deduce that for all $(\alpha,\beta)\in\mathbb N^{2n}$, $(x,\xi),(\tilde x,\tilde\xi)\in\mathbb R^{2n}$ and $0\le t\le 1$,
$$\big\vert(\partial_x^{\alpha}\partial_\xi^{\beta}\nabla_xa)(x+t\tilde x,\xi+t\tilde\xi)\big\vert\le c_{\alpha,\beta}\langle x+t\tilde x\rangle^{\frac gp-\vert\alpha\vert-1}\langle\xi+t\tilde\xi\rangle^{\frac hq-\vert\beta\vert}.$$
Recalling Peetre's inequality, see e.g. \cite{MR2668420} (formula (0.1.2)),
\begin{equation}\label{09092020E3}
	\forall s\in\mathbb R,\exists c_s>0,\forall x,y\in\mathbb R^n,\quad \langle x+y\rangle^s\le c_s\langle x\rangle^s\langle y\rangle^{\vert s\vert},
\end{equation}
we get that for all $(\alpha,\beta)\in\mathbb N^{2n}$, $(x,\xi),(\tilde x,\tilde\xi)\in\mathbb R^{2n}$ and $0\le t\le 1$,
$$\big\vert(\partial_x^{\alpha}\partial_\xi^{\beta}\nabla_xa)(x+t\tilde x,\xi+t\tilde\xi)\big\vert\le c_{\alpha,\beta}c_{\vert\alpha\vert,\vert\beta\vert}\langle x\rangle^{\frac{g-p}p-\vert\alpha\vert}\langle\xi\rangle^{\frac hq-\vert\beta\vert}\langle\tilde x\rangle^{\frac gp+\vert\alpha\vert+1}\langle\tilde\xi\rangle^{\frac hq+\vert\beta\vert}.$$
Plugging this estimate in the definition of $r_1$ yields to \eqref{09092020E1}. This ends the proof of Lemma \ref{09092020L1}.
\end{proof}

In the following lemma, we perform the very same study for the symbol $\frac1{\eta}\langle x\rangle^g + \eta\langle\xi\rangle^h$.

\begin{lem}\label{04092020L1} Let $g,h>0$ be positive real numbers. We consider the symbols
$$H_{\eta}(x,\xi) = \frac1{\eta}\langle x\rangle^g + \eta\langle\xi\rangle^h,\quad\eta>0,\ (x,\xi)\in\mathbb R^{2n}.$$
For all $\eta>0$, there exists some remainders $r_{1,\eta},r_{2,\eta}\in C^{\infty}(\mathbb R^n)$ satisfying
\begin{align}
	& \forall\alpha\in\mathbb N^n,\exists c_{\alpha}>0,\forall\eta>0, \forall x\in\mathbb R^n,\quad \big\vert(\partial_x^{\alpha}r_{1,\eta})(x)\big\vert\le\frac{c_{\alpha}}{\eta}\langle x\rangle^{g-1-\vert\alpha\vert}, \label{11092020E8} \\[5pt]
	& \forall\alpha\in\mathbb N^n,\exists c_{\alpha}>0,\forall\eta>0, \forall\xi\in\mathbb R^n,\quad \big\vert(\partial_{\xi}^{\alpha}r_{2,\eta})(\xi)\big\vert\le c_{\alpha}\eta\langle\xi\rangle^{h-1-\vert\alpha\vert}, \label{03122020E2}
\end{align}
and such that
$$A_{\eta} = H_{\eta}^w + r_{1,\eta}^w + r_{2,\eta}^w,$$
with $A_{\eta}$ the anti-Wick operator with symbol $H_{\eta}$.
\end{lem}

\begin{proof} Mimicking exactly the proof of Lemma \ref{09092020L1}, the two remainder we need to consider are the following ones
\begin{equation}\label{11092020E11}
	r_{1,\eta}(x,\xi) = \frac1{\pi^n}\iint_{\mathbb R^{2n}}\int_0^1\nabla_xH_{\eta}(x+t\tilde x,\xi+t\tilde\xi)\cdot\tilde x\ e^{-\vert\tilde x\vert^2-\vert\tilde\xi\vert^2}\,\mathrm dt\mathrm d\tilde x\mathrm d\tilde\xi,
\end{equation}
and
$$r_{2,\eta}(x,\xi) = \frac1{\pi^n}\iint_{\mathbb R^{2n}}\int_0^1\nabla_{\xi}H_{\eta}(x+t\tilde x,\xi+t\tilde\xi)\cdot\tilde{\xi}\ e^{-\vert\tilde x\vert^2-\vert\tilde\xi\vert^2}\,\mathrm dt\mathrm d\tilde x\mathrm d\tilde\xi.$$
Let us prove that the estimate \eqref{11092020E8} holds. The inequality \eqref{03122020E2} is then obtained by symmetry. First, notice that since the symbol $\nabla_xH_{\eta}$ does not depend on the variable $\xi\in\mathbb R^n$, so does the remainder $r_{1,\eta}$ and we omit this variable in the following. By definition of the symbol $H_{\eta}$ as the sum of Japanese brackets, we know that for all $\alpha\in\mathbb N^n$, there exists a positive constant $c_{\alpha}>0$ such that for all $\eta>0$, $x,\tilde x\in\mathbb R^n$ and $0\le t\le1$,
$$\big\vert(\partial_x^{\alpha}\nabla_xH_{\eta})(x+t\tilde x)\big\vert\le\frac{c_{\alpha}}{\eta}\langle x+t\tilde x\rangle^{g-1-\vert\alpha\vert}.$$
We then deduce from Peetre's inequality \eqref{09092020E3} that for all $\alpha\in\mathbb N^n$, $\eta>0$, $x,\tilde x\in\mathbb R^n$ and $0\le t\le1$,
$$\big\vert(\partial_x^{\alpha}\nabla_xH_{\lambda})(x+t\tilde x)\big\vert\le\frac{c_{\alpha}c_{g,\vert\alpha\vert}}{\eta}\langle x\rangle^{g-1-\vert\alpha\vert}\langle\tilde x\rangle^{g+1+\vert\alpha\vert}.$$
Plugging this estimate in the definition \eqref{11092020E11} of the remainder $r_{1,\eta}$ leads to the estimate \eqref{11092020E8}. This ends the proof of Lemma \ref{04092020L1}.
\end{proof}

We now have all the ingredients required to tackle the proof of the estimate \eqref{17092020E2}.

\begin{prop}\label{17092020P1} Let $g,h>0$ and $p,q\geq1$ be positive real numbers with $1/p+1/q = 1$. For all symbol $a\in S^{2g,2h}_{p,q}$, there exists a positive constant $c>0$ such that for all $\eta>0$ and $v\in\mathscr S(\mathbb R^n)$,
\begin{equation}\label{07092020E1}
	\big\vert\big\langle a^wv,v\big\rangle_{L^2(\mathbb R^n)}\big\vert\le c\bigg(\frac1{\eta}\big\Vert\langle x\rangle^gv\big\Vert^2_{L^2(\mathbb R^n)} + \eta\Vert v\Vert^2_{H^h(\mathbb R^n)}\bigg).
\end{equation}
\end{prop}

\begin{proof} The strategy adopted here follows the one adopted in the proof of \cite{BCC} (Proposition 1). Let $a\in S^{2g,2h}_{p,q}$. Since $1/p+1/q = 1$, we deduce from Young's inequality that there exists a positive constant $c_0>0$ such that for all $\eta>0$ and $(x,\xi)\in\mathbb R^{2n}$,
$$\big\vert a(x,\xi)\big\vert\le c_0\bigg(\frac1{\eta}\langle x\rangle^{2g} + \eta\langle\xi\rangle^{2h}\bigg).$$
For all $\eta>0$, we consider the symbol
$$a^{\pm}_{\eta}(x,\xi) = c_0\bigg(\frac1{\eta}\langle x\rangle^{2g} + \eta\langle\xi\rangle^{2h}\bigg)\pm a(x,\xi),\quad(x,\xi)\in\mathbb R^{2n}.$$
By applying Lemma \ref{09092020L1} and Lemma \ref{04092020L1}, we get the existence of remainders $r_{1,\eta},r_{2,\eta}\in C^{\infty}(\mathbb R^n)$ satisfying that for all $\alpha\in\mathbb N^n$, there exists a positive constant $c_{\alpha}>0$ such that for all $x,\xi\in\mathbb R^n$,
\begin{equation}\label{11092020E9}
	\big\vert(\partial_x^{\alpha}r_{1,\eta})(x)\big\vert\le\frac{c_{\alpha}}{\eta}\langle x\rangle^{2g-1-\vert\alpha\vert},\quad\big\vert(\partial_{\xi}^{\alpha}r_{2,\eta})(\xi)\big\vert\le c_{\alpha}\eta\langle\xi\rangle^{2h-1-\vert\alpha\vert},
\end{equation} 
and also of two more remainders $r_1\in S^{2g-p,2h}_{p,q}$ and $r_2\in S^{2g,2h-p}_{p,q}$, such that
$$A_{a^{\pm}_{\eta}} = (a^{\pm}_\eta)^w + r_{1,\eta}^w + r_{2,\eta}^w\pm r_1^w\pm r_2^w.$$
Since the symbol $a^{\pm}_{\eta}$ is nonnegative and has a polynomial growth, by applying \cite{MR2668420} (Proposition 1.7.6), we obtain that the operator $A_{a^{\pm}_{\eta}}$ is nonnegative, which implies that for all $v\in\mathscr S(\mathbb R^n)$,
\begin{multline*}
 	\langle A_{a^{\pm}}v,v\rangle_{L^2(\mathbb R^n)}  = c_0\bigg\langle\bigg(\frac1{\eta}\langle x\rangle^{2g}+\eta\langle\xi\rangle^{2h}\bigg)^wv,v\bigg\rangle_{L^2(\mathbb R^n)} + \big\langle r_{1,\eta}^wv,v\big\rangle_{L^2(\mathbb R^n)} \\[5pt]
	+ \big\langle r_{2,\eta}^wv,v\big\rangle_{L^2(\mathbb R^n)}\pm\langle a^wv,v\rangle_{L^2(\mathbb R^n)} \pm \big\langle r_1^wv,v\big\rangle_{L^2(\mathbb R^n)}\pm \big\langle r_2^wv,v\big\rangle_{L^2(\mathbb R^n)}\geq0,
\end{multline*}
that is,
\begin{multline}\label{09092020E2}
	\big\vert\big\langle a^wv,v\big\rangle_{L^2(\mathbb R^n)}\big\vert\le c_0\bigg(\frac1{\eta}\big\Vert\langle x\rangle^gv\big\Vert^2_{L^2(\mathbb R^n)}+\eta\Vert v\Vert^2_{H^h(\mathbb R^n)}\bigg)
	+ \big\langle r_{1,\eta}^wv,v\big\rangle_{L^2(\mathbb R^n)} \\[5pt]
	+ \big\langle r_{2,\eta}^wv,v\big\rangle_{L^2(\mathbb R^n)} + \big\vert\big\langle r_1^wv,v\big\rangle_{L^2(\mathbb R^n)}\big\vert + \big\vert\big\langle r_2^wv,v\big\rangle_{L^2(\mathbb R^n)}\big\vert.
\end{multline}
We now have to control all the remainder terms. First, notice that the symbol $\eta r_{1,\eta}$ does not depend on the parameter $\eta>0$ by definition \eqref{11092020E11}. Moreover, we deduce from the estimate \eqref{11092020E9} that 
$$\eta r_{1,\eta}\in S(\langle x\rangle^{2g-1} ; 1,\langle x\rangle),$$
and the associated seminorms are of course independent of the parameter $\eta>0$. By applying the composition formula \eqref{16112020E4}, we get that
$$\langle x\rangle^{-g+\frac12}\ \sharp\ \eta r_{1,\eta}\ \sharp\ \langle x\rangle^{-g+\frac12}\in S(1 ; 1,\langle x\rangle)\subset C^{\infty}_b(\mathbb R^{2n}).$$
Consequently, by applying a quantitative version of the Calder\'on-Vaillancourt's theorem, for which we refer to \cite{MR1696697} (Theorem 1.2), there exists a positive constant $c_{1,1}>0$ such that for all $\eta>0$ and $v\in\mathscr S(\mathbb R^n)$,
$$\big\vert\big\langle\langle x\rangle^{-g+\frac12}\eta r_{1,\eta}^w\langle x\rangle^{-g+\frac12}v,v\big\rangle_{L^2(\mathbb R^n)}\big\vert\le c_{1,1}\Vert v\Vert^2_{L^2(\mathbb R^n)}.$$
Applying this estimate to the function $\langle x\rangle^{g-\frac12}v$, we obtain that for all $\eta>0$ and $v\in\mathscr S(\mathbb R^n)$,
\begin{equation}\label{11092020E10}
	\big\vert\big\langle r_{1,\eta}^wv,v\big\rangle_{L^2(\mathbb R^n)}\big\vert\le\frac{c_{1,1}}{\eta}\big\Vert\langle x\rangle^{g-\frac12}v\big\Vert^2_{L^2(\mathbb R^n)}\le\frac{c_{1,1}}{\eta}\Vert\langle x\rangle^g v\Vert^2_{L^2(\mathbb R^n)}.
\end{equation}
Proceeding similarly, we get that there exists another positive constant $c_{2,1}>0$ such that for all $\eta>0$ and $v\in\mathscr S(\mathbb R^n)$,
\begin{equation}\label{03122020E3}
	\big\vert\big\langle(r_{2,\eta})^wv,v\big\rangle_{L^2(\mathbb R^n)}\big\vert\le c_{2,1}\eta\Vert v\Vert^2_{H^h(\mathbb R^n)}.
\end{equation}
Finally, we have to control the two terms $\langle r_1^w v,v\rangle_{L^2(\mathbb R^n)}$ and $\langle r_2^wv,v\rangle_{L^2(\mathbb R^n)}$. By symmetry, we only focus on $\langle r_1^w v,v\rangle_{L^2(\mathbb R^n)}$. To that end, we proceed by induction. When $2g>p$, we know by the induction assumption, since $r_1\in S^{2g-p,h}_{p,q}$, that there exists a positive constant $c_1>0$ such that for all $\eta>0$ and $v\in\mathscr S(\mathbb R^n)$,
\begin{align*}
	\big\vert\big\langle r_1^wv,v\big\rangle_{L^2(\mathbb R^n)}\big\vert & \le c_1\bigg(\frac1{\eta}\big\Vert\langle x\rangle^{g-\frac p2}v\big\Vert^2_{L^2(\mathbb R^n)} + \eta\Vert v\Vert^2_{H^h(\mathbb R^n)}\bigg) \\[5pt]
	& \le c_1\bigg(\frac1{\eta}\big\Vert\langle x\rangle^gv\big\Vert^2_{L^2(\mathbb R^n)} + \eta\Vert v\Vert^2_{H^h(\mathbb R^n)}\bigg).
\end{align*}
In the other case where $2g\le p$, we have
$$r_1\in S(\langle\xi\rangle^{\frac{2h}q} ; \langle\xi\rangle, 1).$$
Proceeding in the very same way as we did to obtain the estimate \eqref{11092020E10}, we get the existence of a positive constant $c_1'>0$ such that for all $v\in\mathscr S(\mathbb R^n)$,
$$\big\vert\big\langle r_1^wv,v\big\rangle_{L^2(\mathbb R^n)}\big\vert\le c'_1\big\Vert\langle D_x\rangle^{\frac hq}v\big\Vert^2_{L^2(\mathbb R^n)}.$$
Using Plancherel's theorem and Young's inequality ($1/p+1/q = 1$ by assumption) then provides the existence of a positive constant $c_1''>0$ such that for all $\eta>0$ and $v\in\mathscr S(\mathbb R^n)$,
\begin{align}\label{11092020E12}
	\big\vert\big\langle r_1^wv,v\big\rangle_{L^2(\mathbb R^n)}\big\vert
	& \le c_1''\bigg(\frac1{\eta}\Vert v\Vert^2_{L^2(\mathbb R^n)} + \eta\Vert v\Vert^2_{H^h(\mathbb R^n)}\bigg) \\[5pt]
	& \le c_1''\bigg(\frac1{\eta}\Vert\langle x\rangle^gv\Vert^2_{L^2(\mathbb R^n)} + \eta\Vert v\Vert_{H^h(\mathbb R^n)}\bigg). \nonumber
\end{align}
Proceeding similarly, we obtain the existence of another positive constant $c_2>0$ such that for all $\eta>0$ and $v\in\mathscr S(\mathbb R^n)$,
\begin{equation}\label{03122020E4}
	\big\vert\big\langle r_2^wv,v\big\rangle_{L^2(\mathbb R^n)}\big\vert
	\le c_2\bigg(\frac1{\eta}\Vert\langle x\rangle^gv\Vert^2_{L^2(\mathbb R^n)} + \eta\Vert v\Vert_{H^h(\mathbb R^n)}\bigg).
\end{equation}
Plugging the estimates \eqref{11092020E10}, \eqref{03122020E3}, \eqref{11092020E12} and \eqref{03122020E4} in \eqref{09092020E2} provides the estimate \eqref{07092020E1} we aimed at proving.
\end{proof}

To end this section, let us propely derive the estimate \eqref{17092020E1}. Let $a_0\in C^{\infty}(\mathbb R^{2n})$ be the symbol defined in \eqref{17092020E3}. According to \eqref{17092020E3}, \eqref{17092020E5} and Proposition \ref{17092020P1}, there exists a positive constant $c>0$ such that for all $\eta>0$ and $v\in\mathscr S(\mathbb R^n)$,
$$\big\vert\big\langle a_0^wv,v\big\rangle_{L^2(\mathbb R^n)}\big\vert\le c\bigg(\frac1{\eta}\big\Vert\langle x\rangle^{\sigma k}v\big\Vert^2_{L^2(\mathbb R^n)} + \eta\Vert v\Vert^2_{H^m(\mathbb R^n)}\bigg).$$
This proves that the estimate \eqref{17092020E1} actually holds, since
$$\Vert v\Vert^2_{H^m(\mathbb R^n)}\lesssim\Vert v\Vert^2_{\dot H^m(\mathbb R^n)} + \Vert v\Vert^2_{L^2(\mathbb R^n)}.$$
It also ends the proof of the G{\aa}rding type inequality \eqref{06022020E1}.

\section{Appendix}\label{appendix}

\subsection{Gelfand-Shilov spaces}\label{GS} To begin this appendix, let us define and recall basics about Gelfand-Shilov regularity. Given $\mu,\nu>0$ some positive real numbers such that $\mu+\nu\geq1$, we define the Gelfand-Shilov space $S^{\mu}_{\nu}(\mathbb R^n)$, following \cite{MR2668420} (Definition 6.1.1), as the space of Schwartz functions $g\in\mathscr S(\mathbb R^n)$ satisfying that there exist some positive constants $\varepsilon>0$ and $C>0$ such that
$$\begin{array}{ll}
	\forall x\in\mathbb R^n,\quad & \vert g(x)\vert\le Ce^{-\varepsilon\vert x\vert^{\frac1{\nu}}}, \\[7pt]
	\forall\xi\in\mathbb R^n,\quad & \vert \widehat g(\xi)\vert\le Ce^{-\varepsilon\vert\xi\vert^{\frac1{\mu}}},
\end{array}$$
where $\widehat g\in\mathscr S(\mathbb R^n)$ denotes the Fourier transform of the function $g\in\mathscr S(\mathbb R^n)$. We recall from \cite{MR2668420} (Theorem 6.1.6) the basic different equivalent characterizations of the Gelfand-Shilov spaces: 
\begin{enumerate}[label=$(\roman*)$,leftmargin=23pt ,parsep=2pt,itemsep=0pt,topsep=2pt]
\item $g\in S^{\mu}_{\nu}(\mathbb R^n)$.
\item There exists a positive constant $C>1$ such that 
$$\begin{array}{ll}
	\forall x\in\mathbb R^n,\forall\alpha\in\mathbb N^n,\quad & \big\Vert x^{\alpha}g(x)\big\Vert_{L^{\infty}(\mathbb R^n)}\le C^{1+\vert\alpha\vert}\ (\alpha!)^{\nu}, \\[7pt]
	\forall \xi\in\mathbb R^n,\forall\beta\in\mathbb N^n,\quad & \big\Vert\xi^{\beta}\widehat g(\xi)\big\Vert_{L^{\infty}(\mathbb R^n)}\le C^{1+\vert\beta\vert}\ (\beta!)^{\mu}.
\end{array}$$
\item There exists a positive constant $C>1$ such that 
$$\begin{array}{ll}
	\forall x\in\mathbb R^n,\forall\alpha\in\mathbb N^n,\quad & \big\Vert x^{\alpha}g(x)\big\Vert_{L^2(\mathbb R^n)}\le C^{1+\vert\alpha\vert}\ (\alpha!)^{\nu}, \\[7pt]
	\forall x\in\mathbb R^n,\forall\beta\in\mathbb N^n,\quad & \big\Vert\partial^{\beta}_xg(x)\big\Vert_{L^2(\mathbb R^n)}\le C^{1+\vert\beta\vert}\ (\beta!)^{\mu}.
\end{array}$$
\item There exists a positive constant $C>1$ such that
$$\forall(\alpha,\beta)\in\mathbb N^{2n},\quad \big\Vert x^{\alpha}\partial^{\beta}_xg(x)\big\Vert_{L^2(\mathbb R^n)}\le C^{1+\vert\alpha\vert+\vert\beta\vert}\ (\alpha!)^{\nu}\ (\beta!)^{\mu}.$$
\item There exists a positive constant $C>1$ such that
$$\forall(\alpha,\beta)\in\mathbb N^{2n},\quad \big\Vert x^{\alpha}\partial^{\beta}_xg(x)\big\Vert_{L^{\infty}(\mathbb R^n)}\le C^{1+\vert\alpha\vert+\vert\beta\vert}\ (\alpha!)^{\nu}\ (\beta!)^{\mu}.$$
\end{enumerate}
The assumption $\mu+\nu\geq1$ is justified by the following result, coming from the book \cite{MR2668420} anew, which can be read as a version of the Heisenberg's uncertainty principle. It shows that the Gelfand-Shilov class $S^{\mu}_{\nu}(\mathbb R^n)$ as defined above is trivial when $\mu+\nu<1$.

\begin{thm}[Theorem 6.1.10 in \cite{MR2668420}]\label{20112020T1} Any Schwartz function $g\in\mathscr S(\mathbb R^n)$ satisfying that there exist some positive constants $\mu,\nu>0$ with $\mu+\nu<1$, $\varepsilon>0$ and $C>0$ such that
\begin{equation}\tag{$i$}
\begin{array}{ll}
	\forall x\in\mathbb R^n,\quad & \vert g(x)\vert\le Ce^{-\varepsilon\vert x\vert^{\frac1{\nu}}}, \\[7pt]
	\forall\xi\in\mathbb R^n,\quad & \vert \widehat g(\xi)\vert\le Ce^{-\varepsilon\vert\xi\vert^{\frac1{\mu}}},
\end{array}
\end{equation}
is identically equal to zero. Moreover, the same holds when the assumption $(i)$ is replaced by any of the above conditions $(ii)$, $(iii)$, $(iv)$ or $(v)$.
\end{thm}

Gelfand-Shilov regularity can also be defined in terms of exponential decrease in $L^2(\mathbb R^n)$ as shown in the following result whose proof is given in order to make explicit the various implied constants.

\begin{lem}\label{02012020L1} Let $\mu,\nu>0$ be some positive real numbers satisfying $\nu+\mu\geq1$. There exists a positive constant $C>0$ such that for all Schwartz function $g\in\mathscr S(\mathbb R^n)$ satisfying that there exist some positive constants $0<\Lambda_1,\Lambda_2<1$ and $\Lambda_3>0$ such that
\begin{equation}\label{17112019E1}
	\big\Vert e^{\Lambda_1\vert x\vert^{\frac1{\nu}}}g\big\Vert_{L^2(\mathbb R^n)} + \big\Vert e^{\Lambda_2\vert D_x\vert^{\frac1\mu}}g\big\Vert_{L^2(\mathbb R^n)}\le \Lambda_3,
\end{equation}
then $g\in S^{\mu}_{\nu}(\mathbb R^n)$, with the following estimates for the associated seminorms
$$\forall(\alpha,\beta)\in\mathbb N^{2n},\quad\big\Vert x^{\alpha}\partial^{\beta}_xg\big\Vert_{L^2(\mathbb R^n)}\le\frac{C^{\vert\alpha\vert+\vert\beta\vert}}{\Lambda_1^{\nu\vert\alpha\vert}\Lambda_2^{\mu\vert\beta\vert}}\ (\alpha!)^{\nu}\ (\beta!)^{\mu}\ \Lambda_3.$$
\end{lem}

\begin{proof} We refer to \cite{MR2668420} (Subsection 0.3) for the various factorial estimates and estimates involving binomial coefficients used in the following. Let $g\in\mathscr S(\mathbb R^n)$ be a Schwartz function satisfying \eqref{17112019E1} for some $0<\Lambda_1,\Lambda_2<1$ and $\Lambda_3>0$. We first deduce from \eqref{17112019E1} and the estimates 
$$\forall p,q>0,\forall x\geq0,\quad x^pe^{-x^q}\le\bigg(\frac p{eq}\bigg)^{\frac pq},$$
coming from a straightforward study of function, and
$$\forall\alpha\in\mathbb N^n,\quad\vert\alpha\vert^{\vert\alpha\vert}\le e^{\vert\alpha\vert}\vert\alpha\vert!\le (ne)^{\vert\alpha\vert}\alpha!,$$
that for all $\alpha\in\mathbb N^n$,
\begin{equation}\label{15102020E2}
	\big\Vert x^{\alpha}g\big\Vert_{L^2(\mathbb R^n)} 
	= \big\Vert x^{\alpha}e^{-\Lambda_1\vert x\vert^{\frac 1{\nu}}}e^{\Lambda_1\vert x\vert^{\frac1{\nu}}}g\big\Vert_{L^2(\mathbb R^n)}
	\le\bigg(\frac{\nu\vert\alpha\vert}{e\Lambda_1}\bigg)^{\nu\vert\alpha\vert}\Lambda_3
	\le\bigg(\frac{\nu n}{\Lambda_1}\bigg)^{\nu\vert\alpha\vert}\ (\alpha!)^{\nu}\ \Lambda_3.
\end{equation}
The very same arguments and Plancherel's theorem also imply that for all $\beta\in\mathbb N^n$,
\begin{equation}\label{15102020E3}
	\big\Vert \partial^{\beta}_xg\big\Vert_{L^2(\mathbb R^n)} 
	= \big\Vert\partial^{\beta}_xe^{-\Lambda_2\vert D_x\vert^{\frac1{\mu}}}e^{\Lambda_2\vert D_x\vert^{\frac1{\mu}}}g\big\Vert_{L^2(\mathbb R^n)}
	\le\bigg(\frac{\mu\vert\beta\vert}{e\Lambda_2}\bigg)^{\mu\vert\beta\vert}\Lambda_3
	\le\bigg(\frac{\mu n}{\Lambda_2}\bigg)^{\mu\vert\beta\vert}\ (\beta!)^{\mu}\ \Lambda_3.
\end{equation}
Let $(\alpha,\beta)\in\mathbb N^{2n}$ fixed. An integration by parts shows that
$$\big\Vert x^{\alpha}\partial^{\beta}_xg\big\Vert^2_{L^2(\mathbb R^n)} = \big\langle x^{\alpha}\partial^{\beta}_xg,x^{\alpha}\partial^{\beta}_xg\big\rangle_{L^2(\mathbb R^n)}
= (-1)^{\beta}\big\langle\partial^{\beta}_x(x^{2\alpha}\partial^{\beta}_xg),g\big\rangle_{L^2(\mathbb R^n)},$$
while Leibniz' formula provides
\begin{align*}
	\partial^{\beta}_x(x^{2\alpha}\partial^{\beta}_xg) & = \sum_{\gamma\le\beta}\binom{\beta}{\gamma}\partial^{\gamma}_x(x^{2\alpha})\ \partial^{\beta-\gamma}_x(\partial^{\beta}_xg) \\[5pt]
	& = \sum_{\gamma\le\beta\ \gamma\le2\alpha}\binom{\beta}{\gamma}\frac{(2\alpha)!}{(2\alpha-\gamma)!}\ x^{2\alpha-\gamma}\ \partial^{2\beta-\gamma}_xg \\[5pt]
	& = \sum_{\gamma\le\beta\ \gamma\le2\alpha}\binom{\beta}{\gamma}\binom{2\alpha}{\gamma}\gamma!\ x^{2\alpha-\gamma}\ \partial^{2\beta-\gamma}_xg.
\end{align*}
We therefore deduce from Cauchy-Schwarz' inequality that
\begin{equation}\label{15102020E4}
	\big\Vert x^{\alpha}\partial^{\beta}_xg\big\Vert^2_{L^2(\mathbb R^n)}
\le\sum_{\gamma\le\beta\ \gamma\le2\alpha}\binom{\beta}{\gamma}\binom{2\alpha}{\gamma}\gamma!\ \big\Vert x^{2\alpha-\gamma}g\big\Vert_{L^2(\mathbb R^n)}\ \big\Vert\partial^{2\beta-\gamma}_xg\big\Vert_{L^2(\mathbb R^n)},
\end{equation}
with the following estimates coming from \eqref{15102020E2} and \eqref{15102020E3},
\begin{multline}\label{15102020E5}
	\gamma!\ \big\Vert x^{2\alpha-\gamma}g\big\Vert_{L^2(\mathbb R^n)}\ \big\Vert\partial^{2\beta-\gamma}_xg\big\Vert_{L^2(\mathbb R^n)} \\
	\le\bigg(\frac{\nu n}{\Lambda_1}\bigg)^{\nu\vert2\alpha-\gamma\vert}\bigg(\frac{\mu n}{\Lambda_2}\bigg)^{\mu\vert2\beta-\gamma\vert}\ \gamma!\ ((2\alpha-\gamma)!)^{\nu}\ ((2\beta-\gamma)!)^{\mu}\ (\Lambda_3)^2.
\end{multline}
Since $0<\Lambda_1,\Lambda_2<1$, notice that
\begin{equation}\label{15102020E8}
	\bigg(\frac{\nu n}{\Lambda_1}\bigg)^{\nu\vert2\alpha-\gamma\vert}\le\frac{\max(1,\nu n)^{2\nu\vert\alpha\vert}}{\Lambda_1^{2\nu\vert\alpha\vert}}\quad\text{and}\quad\bigg(\frac{\mu n}{\Lambda_2}\bigg)^{\mu\vert2\beta-\gamma\vert}\le\frac{\max(1,\mu n)^{2\mu\vert\beta\vert}}{\Lambda_2^{2\mu\vert\beta\vert}}.
\end{equation}
Moreover, since $\nu+\mu\geq1$, we also get while exploiting the following factorial estimates,
$$\forall\delta,\eta\in\mathbb N^n,\quad \delta!\eta!\le (\delta+\eta)!\le2^{\vert\delta+\eta\vert}\delta!\eta!,$$
that
\begin{align}\label{15102020E6}
	\gamma!\ ((2\alpha-\gamma)!)^{\nu}\ ((2\beta-\gamma)!)^{\mu}
	& \le (\gamma!\ (2\alpha-\gamma)!)^{\nu}\ (\gamma!\ (2\beta-\gamma)!)^{\mu} \\[5pt]
	& \le ((2\alpha)!)^{\nu}\ ((2\beta)!)^{\mu}
	\le 4^{\nu\vert\alpha\vert+\mu\vert\beta\vert}\ (\alpha!)^{2\nu}\ (\beta!)^{2\mu}. \nonumber
\end{align}
We also have from classical results concerning binomial coefficients that
\begin{equation}\label{15102020E7}
	\sum_{\gamma\le\beta\ \gamma\le2\alpha}\binom{\beta}{\gamma}\binom{2\alpha}{\gamma}\le\sum_{\gamma\le\beta}\binom{\beta}{\gamma}\sum_{\gamma\le2\alpha}\binom{2\alpha}{\gamma}
	= 2^{2\vert\alpha\vert+\vert\beta\vert}.
\end{equation}
Finally, we deduce from \eqref{15102020E4}, \eqref{15102020E5}, \eqref{15102020E8}, \eqref{15102020E6} and \eqref{15102020E7} that there exists a positive constant $C>0$ only depending on $\nu>0$ and $\nu>0$ (and not on the function $g$) such that for all $(\alpha,\beta)\in\mathbb N^{2n}$,
$$\big\Vert x^{\alpha}\partial^{\beta}_xg\big\Vert_{L^2(\mathbb R^n)}\le\frac{C^{\vert\alpha\vert+\vert\beta\vert}}{\Lambda_1^{\nu\vert\alpha\vert}\Lambda_2^{\mu\vert\beta\vert}}\ (\alpha!)^{\nu}\ (\beta!)^{\mu}\ \Lambda_3.$$
This ends the proof of Lemma \ref{02012020L1}.
\end{proof}

In the case where the ratio $\mu/\nu\in\mathbb Q$ is a rational number, the Gelfand-Shilov space $S^{\mu}_{\nu}(\mathbb R^n)$ can also be nicely characterized through the decomposition into the basis of eigenfunctions of a large class of anisotropic Shubin operators, whose basic model is the operator $H_{k,m}$ defined in \eqref{19022020E1}, with $k,m\geq1$ two positive integers. Let $(\psi_j)_j$ be an orthonormal basis of $L^2(\mathbb R^n)$ composed of eigenfunctions of the operator $H_{k,m}$. Given a positive real number $a\geq1$, we can characterize the Gelfand-Shilov space $S^{\mu}_{\nu}(\mathbb R^n)$, with 
$$\mu = \frac{ka}{k+m}\quad\text{and}\quad\nu = \frac{ma}{k+m},$$ 
in the following way, according to the result \cite{MR3996060} (Theorem 1.4) by M. Cappiello, T. Gramchev, S. Pilipovi\'c and L. Rodino,
\begin{align}\label{15102020E1}
	g\in S^{\frac{ka}{k+m}}_{\frac{ma}{k+m}}(\mathbb R^n) & \Leftrightarrow\exists\varepsilon>0,\quad\sum_{j=0}^{+\infty}\big\vert\langle g,\psi_j\rangle_{L^2(\mathbb R^n)}\big\vert^2e^{\varepsilon\lambda_j^{\frac{k+m}{2kma}}}<+\infty, \\[5pt]
	& \Leftrightarrow\exists\varepsilon>0,\quad\sup_{j\geq0}\big\vert\langle g,\psi_j\rangle_{L^2(\mathbb R^n)}\big\vert^2e^{\varepsilon\lambda_j^{\frac{k+m}{2kma}}} < +\infty, \nonumber \\[5pt]
	& \Leftrightarrow\exists c>0,\exists\varepsilon>0,\forall j\geq0,\quad\big\vert\langle g,\psi_j\rangle_{L^2(\mathbb R^n)}\big\vert\le ce^{-\varepsilon j^{\frac1{an}}}, \nonumber
\end{align}
where $\lambda_j>0$ denotes the eigenvalue associated with the eigenfunction $\psi_j\in L^2(\mathbb R^n)$. Notice that such a characterization in the case where $\mu/\nu\notin\mathbb Q$ has not been found yet, see \cite{MR3996060} (Section 3).

\subsection{Density in the graph} The purpose of this second subsection is to prove that the Schwartz space $\mathscr S(\mathbb R^n)$ is dense in the domain of the maximal realizations on $L^2(\mathbb R^n)$ of the differential operators equipped with the graph norm. More precisely, by using results of Weyl calculus introduced in Section \ref{secgard}, we aim at proving the following

\begin{prop}\label{07102020P1} Let $p:\mathbb R^{2n}\rightarrow\mathbb C$ be a polynomial and $p^w$ be the associated differential operator equipped with the following domain 
$$D(p^w) = \big\{u\in L^2(\mathbb R^n) : p^wu\in L^2(\mathbb R^n)\big\}.$$
Then, for all $u\in D(p^w)$, there exists a sequence $(u_j)_j$ in $\mathscr S(\mathbb R^n)$ such that
$$\lim_{j\rightarrow+\infty}u_j = u\quad\text{and}\quad \lim_{j\rightarrow+\infty}p^wu_j = p^wu\quad \text{in $L^2(\mathbb R^n)$}.$$
\end{prop}

This proof is based on the symbolic calculus and we will need to use the following approximation result coming from \cite{MR2599384} (Lemma 1.1.3). This proposition is originally stated in the context of standard symbolic calculus, but can easily be adapted in the context of Weyl calculus. This can be done by using \cite{MR2599384} (Proposition 1.1.10) e.g. which makes the link between the standard and the Weyl quantizations.

\begin{prop}[Lemma 1.1.3 in \cite{MR2599384}]\label{09112020P2} Let $(a_j)_j$ be a sequence in $\mathscr S(\mathbb R^{2n})$ being bounded in the space $C^{\infty}_b(\mathbb R^{2n})$ and converging in $C^{\infty}(\mathbb R^{2n})$ to a function $a\in C^{\infty}(\mathbb R^{2n})$. Then, $a$ belongs to $C^{\infty}_b(\mathbb R^{2n})$ and for all $u\in\mathscr S(\mathbb R^n)$,
$$\lim_{j\rightarrow+\infty}a^w_ju = a^wu\quad \text{in $\mathscr S(\mathbb R^n)$}.$$
\end{prop}

Let us now tackle the proof of Proposition \ref{07102020P1}. Let $\chi\in C^{\infty}_0(\mathbb R^{2n})$ be a smooth function satisfying that $0\le\chi\le1$ and $\chi = 1$ on the unit ball $B(0,1)$. For all $j\geq1$, we consider the compactly supported function $\chi_j$ defined by $\chi_j(x,\xi) = \chi(x/j,\xi/j)$ for all $(x,\xi)\in\mathbb R^{2n}$. Since the function $\chi$ is compactly supported, we get that
\begin{equation}\label{23052018E1}
	\forall j\geq1,\quad \chi_j^w : L^2(\mathbb R^n)\rightarrow\mathscr S(\mathbb R^n).
\end{equation}
Let us begin by checking the following property
\begin{equation}\label{09112020E1}
	\forall u\in L^2(\mathbb R^n),\quad \lim_{j\rightarrow+\infty}\chi^w_ju = u\quad \text{in $L^2(\mathbb R^n)$}.
\end{equation}
Let $u\in L^2(\mathbb R^n)$. We consider $\varepsilon>0$. By density, there exists a Schwartz function $v\in\mathscr S(\mathbb R^n)$ such that $\Vert u-v\Vert_{L^2(\mathbb R^n)}\le\varepsilon$. We can write
\begin{equation}\label{30112020E2}
	\chi_j^wu -u = \chi_j^w(u-v) + \chi_j^wv-v + v-u.
\end{equation}
By construction, the sequence of Schwartz symbols $(\chi_j)_j$ is bounded on $C^{\infty}_b(\mathbb R^{2n})$ and converges to $1$ in $C^{\infty}(\mathbb R^{2n})$. Proposition \ref{09112020P2} therefore implies that
\begin{equation}\label{30112020E3}
	\lim_{j\rightarrow+\infty}\chi_j^wv = v\quad \text{in $L^2(\mathbb R^n)$}.
\end{equation}
Moreover, the boundedness of the sequence $(\chi_j)_j$ in $C^{\infty}_b(\mathbb R^{2n})$ and a quantitative version of the Calder\'on-Vaillancourt's theorem, see e.g. \cite{MR1696697} (Theorem 1.2), provide the existence of a positive constant $c>0$ such that
\begin{equation}\label{30112020E4}
	\forall j\geq1,\quad\big\Vert\chi_j^w\big\Vert_{\mathcal L(L^2(\mathbb R^n))}\le c,
\end{equation}
where $\mathcal L(L^2(\mathbb R^n))$ denotes the space of bounded operators on $L^2(\mathbb R^n)$. We therefore deduce from \eqref{30112020E2}, \eqref{30112020E3} and \eqref{30112020E4} that 
$$\limsup_{j\rightarrow+\infty}\big\Vert\chi_j^wu -u\big\Vert_{L^2(\mathbb R^n)}\le(c+1)\varepsilon.$$
Since $\varepsilon>0$ is arbitrary, this proves that \eqref{09112020E1} actually holds.

Now, let us consider $u\in D(p^w)$ and set $u_j = \chi_j^wu$ for all $j\geq1$. According to \eqref{23052018E1} and \eqref{09112020E1}, $(u_j)_j$ is a sequence of Schwartz functions that converges to $u$ in $L^2(\mathbb R^n)$. Since $p^wu\in L^2(\mathbb R^n)$ by definition of the domain $D(p^w)$, we can apply once again \eqref{09112020E1} to get that
$$\lim_{j\rightarrow+\infty}\chi_j^wp^wu = p^wu\quad \text{in $L^2(\mathbb R^n)$}.$$
If the operators $\chi_j^w$ and $p^w$ were commutative, Proposition \ref{07102020P1} would be proven. It is not the case but to conclude, it is sufficient to check that 
\begin{equation}\label{09052018E9}
	\lim_{j\rightarrow+\infty}\big[p^w,\chi_j^w\big]u = 0\quad \text{in $L^2(\mathbb R^n)$}.
\end{equation}
Let $j\geq1$ and $a_j$ be the Weyl symbol of the commutator $[p^w,\chi_j^w]$. The strategy to establish \eqref{09052018E9} is to check that the sequence $(a_j)_j$ is bounded in $C^{\infty}_b(\mathbb R^{2n})$ and converges to 0 in $C^{\infty}(\mathbb R^{2n})$. Once this is done, Proposition \ref{09112020P2} implies that 
$$\forall v\in\mathscr S(\mathbb R^n),\quad \lim_{j\rightarrow+\infty}\big[p^w,\chi_j^w\big]v = 0\quad \text{in $L^2(\mathbb R^n)$},$$
and the very same arguments as the ones used to obtain \eqref{09112020E1} show that this convergence holds for all $v\in L^2(\mathbb R^n)$, and in particular for the function $u$. To that end, we will derive a formula for the symbol $a_j$. By using the same strategy as in the beginning of the proof of Lemma \ref{10092020L1}, that is, by using the fact that the symbol $p$ is a polynomial (of degree $d\geq1$ say) and the composition formula \eqref{16112020E3}, we get that the symbol $a_j$ is explicitly given by
\begin{equation}\label{09112020E2}
	a_j = p\ \sharp\ \chi_j -  \chi_j\ \sharp\ p= 2\sum_l\frac1{(2i)^l}\sum_{\vert\alpha\vert + \vert\beta\vert = l}\frac{(-1)^{\vert\beta\vert}}{\alpha!\beta!}(\partial^{\alpha}_{\xi}\partial^{\beta}_xp)(\partial^{\beta}_{\xi}\partial^{\alpha}_x\chi_j),
\end{equation}
the sum being taken over all the odd integers $l$ satisfying $1\le l\le d$. Notice that by definition of the cutoff function $\chi_j$,
$$\forall(\alpha,\beta)\in\mathbb N^{2n}, \forall(x,\xi)\in\mathbb R^{2n},\quad(\partial^{\alpha}_{\xi}\partial^{\beta}_x\chi_j)(x,\xi) = j^{-\vert\alpha+\beta\vert}(\partial^{\alpha}_{\xi}\partial^{\beta}_x\chi)(x/j,\xi/j).$$
This combined with Leibniz' formula shows that the sequence $(a_j)_j$ is bounded in $C^{\infty}_b(\mathbb R^{2n})$.
Moreover, the function $\chi$ is constant equal to $1$ in a neighborhood of $0$, which implies that for all $(\alpha,\beta)\in\mathbb N^{2n}$ satisfying $\vert\alpha+\beta\vert\geq1$ and $(x,\xi)\in\mathbb R^{2n}$,
$$\lim_{j\rightarrow+\infty}(\partial^{\alpha}_{\xi}\partial^{\beta}_x\chi_j)(x,\xi) = 0.$$
By using Leibniz's formula anew, we also deduce that the sequence $(a_j)_j$ converges to $0$ in $C^{\infty}(\mathbb R^{2n})$, the integers $l$ involved in \eqref{09112020E2} satisfying $1\le l\le d$. This ends the proof of \eqref{09052018E9} as announced, and therefore the one of Proposition \ref{07102020P1}.


\begin{thebibliography}{10}
\bibitem{Ag1} S. \textsc{Agmon}, \textit{On exponential decay of solutions of second order elliptic equation in unbounded domains}, Proc. A. Pleijel Conf., Uppsala, September 1979.
\bibitem{Ag2} S. \textsc{Agmon}, \textit{Lectures on exponential decay of solutions of second-order elliptic equations}, Mathematical Notes 29, Princeton University Press, 1982.
\bibitem{A} P. \textsc{Alphonse}, \textit{R\'egularit\'e des solutions et contr\^olabilit\'e d'\'equations d'\'evolution associ\'ees \`a des op\'erateurs non-autoadjoints}, PhD Thesis, Universit\'e de Rennes 1 (2020).
\bibitem{AB} P. \textsc{Alphonse}, J. \textsc{Bernier}, \textit{Smoothing properties of fractional Ornstein-Uhlenbeck semigroups and null-controllability}, Bull. Sci. Math. 165 (2020), 102914.
\bibitem{AB2} P. \textsc{Alphonse}, J. \textsc{Bernier}, \textit{Polar decomposition of semigroups generated by non-selfadjoint quadratic differential operators and regularizing effects}, to appear in Ann. Scient. Ec. Norm. Sup, \href{https://arxiv.org/abs/1909.03662}{arXiv:1909.03662}.
\bibitem{AM} P. \textsc{Alphonse}, J. \textsc{Martin}, \textit{Stabilization and approximate null-controllability for a large class of diffusive equations from thick control supports}, ESAIM Control Optim. Calc. Var. 28 (2022), no.16, 30 pp.
\bibitem{BEPS} K. \textsc{Beauchard}, M. \textsc{Egidi}, K. \textsc{Pravda-Starov}, \textit{Geometric conditions for the null-controllability of hypoelliptic quadratic parabolic equations with moving control supports},  C. R. Math. Acad. Sci. Paris 358 (2020), no. 6, pp. 651-700.
\bibitem{BJKPS} K. \textsc{Beauchard}, P. \textsc{Jaming}, K. \textsc{Pravda-Starov}, \textit{Spectral inequality for finite combinations of {H}ermite functions and null-controllability of hypoelliptic quadratic equations}, Studia Math. 260 (2021), pp. 1-43.
\bibitem{MR3732691} K. \textsc{Beauchard}, K. \textsc{Pravda-Starov}, \textit{Null-controllability of hypoelliptic quadratic differential equations}, J. \'Ec. polytech. Math. 5 (2018), pp. 1-43.
\bibitem{BCC} J. \textsc{Bernier}, F. \textsc{Casas}, N. \textsc{Crouseilles}, \textit{A note on some microlocal estimates used to prove the convergence of splitting methods relying on pseudo-spectral discretizations}, \href{https://hal.archives-ouvertes.fr/hal-02929869}{hal-02929869}.
\bibitem{BBR} P. \textsc{Boggiatto}, E. \textsc{Buzano}, L. \textsc{Rodino}, \textit{Global Hypoellipticity and Spectral Theory}, Mathematical Research, Vol.
92, Akademie Verlag, Berlin (1996).
\bibitem{MR1696697} A. \textsc{Boulkhemair}, \textit{$L^2$ estimates for Weyl quantization}, J. Funct. Anal. 165 (1999), no.1, pp. 173-204.
\bibitem{MR2747070} M. \textsc{Cappiello}, T. \textsc{Gramchev}, L. \textsc{Rodino}, \textit{Entire extensions and exponential decay for semilinear elliptic equations}, J. Anal. Math. 111 (2010), pp. 339-367.
\bibitem{MR3996060} M. \textsc{Cappiello}, T. \textsc{Gramchev}, S. \textsc{Pilipovic}, L. \textsc{Rodino}, \textit{Anisotropic Shubin operators and eigenfunction expansions in Gelfand-Shilov spaces}, J. Anal. Math. 138 (2019), no. 2, pp. 857-870.
\bibitem{CDR} M. \textsc{Chatzakou}, J. \textsc{Delgado}, M. \textsc{Ruzhansky}, \textit{On a class of anharmonic oscillators}, J. Math. Pures Appl. 153 (2021), pp. 1-29.
\bibitem{MR2302744} J-M. \textsc{Coron}, \textit{Control and nonlinearity}, Mathematical Surveys and Monographs, Vol. 136, American Mathematical Society, Providence, RI (2007).
\bibitem{MR2984079} T. \textsc{Duyckaerts}, L. \textsc{Miller}, \textit{Resolvent conditions for the control of parabolic equations}, J. Funct. Anal. 11 (2012), pp. 3641-3673.
\bibitem{MR3816981} M. \textsc{Egidi}, I. \textsc{Veseli\'c}, \textit{Sharp geometric condition for null-controllability of the heat equation on $\mathbb R^d$ and consistent estimates on the control cost}, Arch. Math. (Basel) 111 (2018), no. 1, pp. 85-99.
\bibitem{H} B. \textsc{Helffer}, \textit{Th\'eorie spectrale pour des op\'erateurs globalement elliptiques. (Spectral theory for globally elliptic operators) With an English summary}, Ast\'erisque, Vol. 12, Soci\'et\'e Math\'ematique de France, 1984.
\bibitem{HR1} B. \textsc{Helffer}, D. \textsc{Robert}, \textit{Asymptotique des niveaux d'\'energie pour des hamiltoniens a un degr\'e de libert\'e}, Duke Math. J. 49 (1982), no. 4, pp. 853-868.
\bibitem{HR2} B. \textsc{Helffer}, D. \textsc{Robert}, \textit{Propri\'etes asymptotiques du spectre d'op\'erateurs pseudodifferentiels sur $\mathbb R^n$}, Comm. P.D.E. (1982), no. 7, pp. 795-882.
\bibitem{MR2304165} L. \textsc{Hörmander}, \textit{The analysis of linear partial differential operators}, vol. III, Springer Verlag, Berlin (1985).
\bibitem{HWW} S. \textsc{Huang}, G. \textsc{Wang}, M. \textsc{Wang}, \textit{Characterizations of stabilizable sets for some parabolic equations in $\mathbb R^n$}, J. Differential Equations 272 (2021), pp. 255-288.
\bibitem{K} A. \textsc{Koenig}, \textit{Contrôlabilité de quelques équations aux dérivées partielles paraboliques peu diffusives}, Ph.D thesis, Université Côte d'Azur (2019).
\bibitem{Ko} A. \textsc{Koenig}, \textit{Lack of null-Controllability for the fractional heat equation and related equations}, SIAM J. Control Optim. 58 (2020), no. 6, pp. 3130-3160.
\bibitem{MR1840110} O. \textsc{Kovrijkine}, \textit{Some results related to the Logvinenko-Sereda theorem}, Proc. Amer. Math. Soc. 129 (2001), no. 10, pp. 3037-3047.
\bibitem{MR1312710} G. \textsc{Lebeau}, L. \textsc{Robbiano}, \textit{Contr\^ole exact de l'\'equation de la chaleur}, Comm. Partial Differential Equations 20 (1995), no. 1-2, pp. 335-356.
\bibitem{MR2599384} N. \textsc{Lerner}, \textit{Metrics on the phase space and non-selfadjoint pseudo-differential operators}, Pseudo-Differential Operators, Theory and Applications, Vol. 3, Birkh\"auser Verlag, Basel (2010).
\bibitem{LMPSX1} N. \textsc{Lerner}, Y. \textsc{Morimoto}, K. \textsc{Pravda-Starov}, C.-J. \textsc{Xu}, \textit{Gelfand-Shilov and Gevrey smoothing effect for the spatially inho-mogeneous non-cutoff Kac equation}, J. Funct. Anal. 269 (2015), pp. 459-535.
\bibitem{LMPSX2} N. \textsc{Lerner},  Y. \textsc{Morimoto},  K. \textsc{Pravda-Starov},  C.-J. \textsc{Xu}, \textit{Phase space analysis and functional calculus for the linearized Landau and Boltzmann operators}, Kinet. Relat. Models, 6 (2013), no. 3, pp. 625-648.
\bibitem{LMPSX3} N. \textsc{Lerner}, Y. \textsc{Morimoto}, K. \textsc{Pravda-Starov}, C.-J. \textsc{Xu}, \textit{Spectral and phase space analysis of the linearized non-cutoff Kac collision operator}, J. Math. Pures Appl. 100 (2013), no. 6, pp. 832-867.
\bibitem{MPS} J.\textsc{Martin}, K.\textsc{Pravda-Starov}, \textit{Spectral inequalities for combinations of Hermite functions and null-controllability for evolution equations enjoying Gelfand-Shilov smoothing effects}, J. Inst. Math. Jussieu (2022), pp. 1-50.
\bibitem{MPS2} J. \textsc{Martin}, K. \textsc{Pravda-Starov}, \textit{Geometric conditions for the exact controllability of fractional free and harmonic Schr\"odinger equations}, J. Evol. Equ. 21 (2021), no. 1, pp. 1059-1087.
\bibitem{M} L. \textsc{Miller}, \textit{Unique continuation estimates for sums of semiclassical eigenfunctions and null-controllability from cones}, preprint (2008), \href{https://hal.archives-ouvertes.fr/hal-00411840}{hal-00411840}.
\bibitem{MR2679651} L. \textsc{Miller}, \textit{A direct Lebeau-Robbiano strategy for the observability of heat-like semigroups}, Discrete Contin. Dyn. Syst. Ser. B 14 (2010), no. 4, pp. 1465-1485.
\bibitem{MR2668420} F. \textsc{Nicola}, L. \textsc{Rodino}, \textit{Global pseudo-differential calculus on Euclidean spaces}, Pseudo-Differential Operators, Theory and Applications, Vol. 4, Birkh\"auser Verlag, Basel (2010).
\bibitem{R} D. \textsc{Robert}, \textit{Autour de l'approximation semi-classique. (French) (On semiclassical approximation)}, Progress in Mathematics, Vol. 68, Birkh\"auser, Boston, MA, 1987.
\bibitem{S} M. \textsc{Shubin}, \textit{Pseudo differential operators and spectral theory}, Springer Series in Soviet Mathematics, Springer Verlag, Berlin (1987).
\bibitem{TW} M. \textsc{Tucsnak}, G. \textsc{Weiss}, \textit{Observation and control for operator semigroups}, Birkhäuser Advanced Texts: Basler Lehrbücher. Birkhäuser Verlag, Basel (2009).
\bibitem{V} A. \textsc{Voros}, \textit{Oscillator quartic et methodes semi-classiques}, Seminaire Goulaouic-Schwartz. Ecole Polytech., Palaiseau 6 (1980).
\bibitem{WZ} G. \textsc{Wang}, M. \textsc{Wang}, C. \textsc{Zhang}, Y. \textsc{Zhang}, \textit{Observable set, observability, interpolation inequality and spectral inequality for the heat equation in $\mathbb R^n$}, J. Math. Pures Appl. (9) 126 (2019), pp. 144-194.
\end{thebibliography}
\end{document}